\documentclass[12pt]{article}

\usepackage{amssymb}
\usepackage{amsmath,amsfonts,amsthm,mathtools}
\usepackage{tabularx}
\usepackage{xcolor}
\usepackage{authblk}
\usepackage{xurl}
\usepackage[backend=biber,eprint=false,date=year,style=alphabetic,firstinits=true,url=false,maxbibnames=99]{biblatex}
\renewbibmacro{in:}{%
  \ifentrytype{article}{}{\printtext{\bibstring{in}\intitlepunct}}}

\usepackage{hyperref}
\hypersetup{breaklinks=true}
  
\usepackage{cleveref}
\usepackage{enumerate}
\newcommand{\N}{\mathbb{N}}
\newcommand{\R}{\mathbb{R}}

\newcommand{\C}{\mathbb{C}}
\newcommand{\K}{\mathbb{K}}

\newcommand{\calM}{\mathcal{M}}
\newcommand{\calN}{\mathcal{N}}
\newcommand{\calH}{\mathcal{H}}
\newcommand{\calS}{\mathcal{S}}
\newcommand{\calB}{\mathcal{B}}
\newcommand{\calK}{\mathcal{K}}

\DeclareMathOperator{\dist}{dist}
\DeclareMathOperator{\cl}{cl}

\newcommand{\setsep}{\;;\;}

\DeclareMathOperator{\diag}{diag} % diagonalization
\DeclareMathOperator{\ecc}{ecc} % eccentricity
\DeclareMathOperator{\tr}{tr}
\newcommand{\dd}{\mathrm{d}}
\newcommand{\iu}{\mathrm{i}}
\newcommand{\norm}[1]{\left\lVert#1\right\rVert} % \norm{x} is the norm of x
\newcommand{\abs}[1]{\left\lvert#1\right\rvert} % \abs{x} is the absolute value of x
\providecommand{\keywords}[1]{\textit{Keywords:} #1}
\providecommand{\subjclass}[1]{\textit{MSC2020:} #1}

\newtheorem{thm}{Theorem}
\newtheorem{lem}[thm]{Lemma}
\newtheorem{defn}[thm]{Definition}
\newtheorem{prop}[thm]{Proposition}

\newtheorem{example}[thm]{Example}

\newtheorem{cor}[thm]{Corollary}
\newtheorem{rem}[thm]{Remark}
\newtheorem{con}[thm]{Conjecture}

\newcolumntype{L}[1]{>{\raggedright\arraybackslash}p{#1}} 
\newcolumntype{C}[1]{>{\centering\arraybackslash}p{#1}}
\newcolumntype{R}[1]{>{\raggedleft\arraybackslash}p{#1}}

\title{The joint spectral radius is pointwise Hölder continuous}
\author{Jeremias Epperlein\footnote{jeremias.epperlein@uni-passau.de} }
\author{Fabian Wirth\footnote{fabian.lastname@uni-passau.de}}

\affil{University of Passau, Department of Computer Science and
  Mathematics, 94032 Passau, Germany}

\begin{document}
\maketitle
\begin{abstract}
    We show that the joint spectral radius is pointwise Hölder continuous. In addition, the joint spectral radius is locally Hölder continuous for $\varepsilon$-inflations. In the two-dimensional case, local Hölder continuity holds on the matrix sets with positive joint spectral radius. 
\end{abstract}


\section{Introduction}
\label{sec:intro}

The joint spectral radius of a set of matrices was introduced by Rota and Strang in \cite{RotaStra60}. For the last 20 years it has received considerable attention and a considerable body of results has been obtained in this area, see the monograph \cite{Jungers}, the milestone papers \cite{BergWang92,Gurv95,lagariasFinitenessConjectureGeneralized1995,bousch2002asymptotic, hareExplicitCounterexampleLagarias2011}, the surveys \cite{margaliot2006stability,shorten2007stability} and references therein.

One of the earliest questions concerned the regularity of the joint spectral radius as a function of the data. Continuity of the joint spectral radius was shown in \cite{Bara88}. Local Lipschitz continuity on the set of irreducible compact subsets of $\K^{d\times d},\K=\R,\C,$ was obtained in \cite{Wirt02}. A different proof which also yields concrete estimates for the Lipschitz constants is due to Kozyakin, \cite{KOZYAKIN201012}. The property was extended to the set of positive inclusions with a strongly connected system graph in \cite{mason2014extremal}.
The question remained whether the joint spectral radius is more regular than merely continuous on the space $\mathcal{H}(d)$ of all compact subsets of $\K^{d\times d}$ (endowed with the Hausdorff metric). A natural candidate for this type of regularity is Hölder continuity, as it is well known that the spectral radius as a function of a single matrix is locally $\frac{1}{d}$-Hölder on $\K^{d\times d}$. In this paper, we present a collection of results towards the goal of obtaining similar statements for the joint spectral radius as a function on the Hausdorff metric space of compact subsets of $\C^{d\times d}$. Jean-Michel Coron first suggested to the second author during a private conversation in 2001 that the joint spectral radius should be
locally Hölder continuous of degree $\frac{1}{d}$.
Whether or not this is true, is still unknown to us. However, we collect a number of results that suggest that this statement is close to the truth. We will show that
\begin{enumerate}
    \item the joint spectral radius is pointwise Hölder continuous, and for finite matrix sets the constant is $\frac{1}{d+\varepsilon}$, for any $\varepsilon>0$; 
    \item if $d=2$ the joint spectral radius is locally Hölder continuous of order $\frac{1}{6}$ on the set of compact matrix sets with nonzero joint spectral radius,
    \item for arbitrary $d$, the joint spectral radius is locally Hölder continuous of order $\frac{1}{d+1}$ when restricted to $\varepsilon$-inflations of a set $\calM$ by norm balls in a linear subspace of $\C^{d\times d}$.
\end{enumerate}

While the joint spectral radius describes 
the maximal growth rate of products from a set of matrices,
the top Ljapunov exponent describes the average growth
rate of a set of matrices sampled according to some
probability measure.
There has been a resurgence of activity regarding 
the continuity properties of these Lyapunov exponents.
While the continuous dependence of the Lyapunov exponents
of a set of invertible matrices
on the measure (when endowed with
the right topology taking into account the support of the measures)
in the irreducible case was already shown in the classical paper
of Furstenberg and Kifer \cite{furstenbergRandomMatrixProducts1983e} and independently Hennion \cite{hennionLoiGrandsNombres1984},
it took until 2023 for 
Avila, Eskin and Viana \cite{avilaContinuityLyapunovExponents2023}
to also settle the reducible case.
For stronger regularity properties of these Lyapunov exponents 
in various special cases see e.g.
\cite{peresDomainsAnalyticContinuation1992, hennionDerivabiliteGrandExposant1991,
duarteHolderContinuityLyapunov2022,
bezerraUpperBoundRegularity2023}.

The paper is organized as follows. In the following Section~\ref{sec:problemstatement} we present a formal problem statement and we compare our results to the regularity properties we conjecture for the joint spectral radius. In Section~\ref{sec:preliminaries} we recall some fundamental inequalities we require, recall definitions concerning Hölder continuity and discuss some related properties. Next we discuss some properties of reducible matrix sets and norms that are associated to the reducibility structure in Section~\ref{sec:norms}.
Section~\ref{sec:increasing} treats $\varepsilon$-inflation as a first particular case. By $\varepsilon$-inflation we mean the addition of $\varepsilon$-balls (in the Minkowski sense) to a given compact set of matrices $\calM$ and the dependence of the joint spectral radius on $\varepsilon$. We show that this dependence is pointwise Hölder continuous with an exponent depending on the irreducibility index of $\calM$ and locally Hölder continuous with a slightly worse exponent. This result is used in Section~\ref{sec:pointwise}, where we show that the joint spectral radius is pointwise Hölder continuous. 
Interestingly, the upper bound for the Hölder continuity is the
optimal $\frac{1}{d}$ that we obtain from the $\varepsilon$-inflation
result. The lower bound on the other hand turns out to be much more intricate than in the case of the ordinary spectral radius. In \Cref{prop:specrad-lowerbound} we point out that for the spectral radius standard perturbation theory yields linear lower bounds. In the general case, we are only able to show the Hölder constant $\alpha=\frac{1}{d^2+d}$. Interestingly, a result by Morris, \cite{morris2010rapidly}, by which 
the spectral radii of finite products approximate the joint spectral at fast rates, can be used to obtain improved - almost optimal - Hölder constants. But Morris' result is only applicable to finite matrix sets. 
In Section~\ref{sec:two-dim} we discuss local Hölder continuity. Here our results are much less complete. The methods we use allow only for a proof on the set of compact matrices with positive joint spectral radius.
We conjecture that on this set there are locally uniform bounds for the exponential-polynomial growth of trajectories. We show that under this condition the joint spectral radius is locally Hölder continuous on the set of compact matrix sets with strictly positive joint spectral radius.
In the two-dimensional case this conjecture can be verified. In Section~\ref{sec:continuous} we briefly state the consequences of our results for the continuous-time version of the joint spectral radius. Here it is automatic that the exponential growth rate corresponding to a compact set of matrices is positive, which leads to slightly more satisfying statements. We conclude in Section~\ref{sec:conclusion}.

\section{Problem Statement}
\label{sec:problemstatement}

Let $\N$ be the set of natural numbers including $0$. The real and complex field are denoted by $\R$, $\C$ and we set $\R_{\geq0}:=[0,\infty)$. The Euclidean norm on $\R^d, \C^d$ is denoted by $\norm{\cdot}_2$ and this also denotes the induced operator norm, i.e., the spectral norm on $\R^{d\times d},\C^{d\times d}$. In the following, $\norm{\cdot}$ denotes an arbitrary fixed vector norm as well as the induced operator norm.

Let $\K=\R,\C$, $d\geq 1$. For a bounded, nonempty set of matrices $\mathcal{M}\subseteq \K^{d\times d}$ we consider the set of arbitrary products of length $k$ defined by 
\begin{equation*}
    {\calS}_k(\calM):= \{ A_{k-1} \cdots A_{0} \setsep A_j \in {\calM}, j=0,\ldots,k-1 \} \,.
\end{equation*}
The joint spectral radius of $\mathcal{M}$ is defined as 
\begin{equation}
\label{eq:jsr}
    \rho(\mathcal{M)} := \lim_{k\to \infty} \sup \{ \norm{S} \setsep S\in {\calS}_k(\calM) \}^{1/k},
\end{equation}
and by the Berger-Wang theorem, \cite{BergWang92}, \cite{Elsn95} it is known that the quantity may be equivalently expressed as
\begin{equation}
\label{eq:Berg-Wang-jsr}
    \rho(\mathcal{M)} := \limsup_{k\to \infty} \sup \{ \rho(S) \setsep S\in {\calS}_k(\mathcal{M}) \}^{1/k},
\end{equation}
where (with slight abuse of notation) $\rho(A)$ for a matrix $A\in\K^{d\times d}$ denotes the standard spectral radius.
It is known that taking the closure of $\mathcal{M}$ does not change the value of the joint spectral radius. Thus we will assume that $\mathcal{M}$ is compact from now on.
With this convention $\rho$ is a map defined on
 $\mathcal{H}(d)$, the metric space of compact subsets of $\K^{d\times d}$ endowed with the Hausdorff metric. Recall that this metric is defined by
 \begin{equation}
     d(\calM,\calN) := \max \left\{ 
     \max \left\{ d(A,\calN) \setsep A \in \calM \right\}, 
     \max \left\{ d(B,\calM) \setsep B \in \calN \right\}
     \right\},
 \end{equation}
 where $d(A,\calN) := \min \{ \norm{A-B}_2 \setsep B \in \calN\}$ is the standard distance of a point to a set.
The space $\calH(d)$ is locally compact, separable and complete.
 
 For the topological subspace 
\begin{equation}
    \mathcal{H}_I(d) := \{ \mathcal{M}\in \mathcal{H}(d) \setsep \mathcal{M} \text{ is irreducible }\}
\end{equation}
it is known that
\begin{equation}
    \rho : \mathcal{H}_I(d) \to \R_{\geq0}, \quad \mathcal{M}\mapsto \rho(\mathcal{M})
\end{equation}
is locally Lipschitz continuous, \cite{Wirt02}. In this paper we investigate further regularity properties of the map
\begin{equation}
    \rho : \mathcal{H}(d) \to \R_{\geq0}, \quad \mathcal{M}\mapsto \rho(\mathcal{M}).
\end{equation}

For the spectral radius $\rho: \C^{d \times d} \to \R_{\geq0}$ the regularity theory in the large is well understood. A
theorem of Elsner (see \cite[Theorem IV.1.3]{stewartMatrixPerturbationTheory1990}) shows
 that the spectral radius on $\C^{d\times d}$
is locally Hölder continuous of order $1/d$.
\begin{thm}[Elsner]
    \label{thm:elsner}
  For $A,B \in \C^{d \times d}$ we have
  \begin{align*}
    \abs{\rho(A) - \rho(B)} \leq
   (\norm{A}_2+\norm{B}_2)^{(d-1)/d} \norm{A-B}_2^{1/d}.
  \end{align*}
\end{thm}

Maybe less familiar is the following locally Lipschitz lower bound for the spectral radius, which actually follows from classic results of analytic perturbation theory. While this is
probably well-known to perturbation theorists,
we were not able to find an explicit or implicit mention of this result in the literature.

\begin{prop}
\label{prop:specrad-lowerbound}
  Let $A \in \C^{d \times d}$.
  Then there are constants $\Gamma, \varepsilon_0 >0$ such that 
  \begin{align*}
    \rho(A + \varepsilon B) \geq \rho(A) - \Gamma \varepsilon
  \end{align*}
  for all $B \in \C^{d \times d}, \norm{B}_2 \leq 1$ and $\varepsilon \in (0,\varepsilon_0)$.
\end{prop}

For the sake of completeness we include a proof of this result in the appendix, see \ref{sec:App}. 

We expect that the regularity of the joint spectral radius does not behave significantly differently to that of the ordinary spectral radius. In view of the two previous results, the following conjecture summarizes our expectations 
concerning provable Hölder regularity properties of the joint spectral radius.
In Figure~\ref{fig:overview} we give an overview and comparison of the results that we are able to prove compared to the conjectures. For the pointwise Hölder case the situation is more favorable, as there our results are closer to the conjectures up to the specific constants, whereas as far as the locally Hölder case is concerned there are still significant gaps between what we can prove and what we would like to prove.

\newpage

\begin{con}
  \label{conjecture}$ $
  \begin{description}
  \item[L1/P1] \textbf{(Local/pointwise Hölder continuity)} 
  \label{conj:L1hoelderbound}The joint
    spectral radius is locally (pointwise) Hölder continuous with
    exponent $\alpha>0$ on the space $\mathcal{H}(d)$ of compact
    subsets of $\C^{d\times d}$ endowed with the Hausdorff
    metric. More specifically, $\alpha = 1/d$.
  \item[L2/P2] 
  \label{conj:L2lowerhoelderbound}\textbf{(Local/pointwise Hölder lower bounds)} There is
    an $\alpha>0$ such that for every $\mathcal{M} \in \mathcal{H}(d)$
    there are constants $\tau,\varepsilon_0>0$ such that
    $d_H(\mathcal{M},\mathcal{N})< \varepsilon < \varepsilon_0$
    implies
    \begin{equation}
        \rho(\mathcal{N}) \geq \rho(\mathcal{M}) - \tau \varepsilon^\alpha.
    \end{equation}
    For $L2$ we assume that $\tau,\varepsilon_0$ can be chosen
    uniformly on compact subsets of $\mathcal{H}(d)$. More
    specifically, $\alpha = 1$ in the pointwise case and
    $\alpha = 1/d$ in the local case.
  \item[L3] \textbf{(locally uniform trajectory bounds)}
    \label{conj:L3bound}
    There is a constant $\Theta_d>0$ depending only on the dimension
    $d$ and the norm $\norm{\cdot}$ such that
    for all $\calM\in \calH(d)$ with $\rho(\calM)=1$ we have for all $k\geq 1$
    \[\norm{S} \leq \Theta_d(Lk)^{d-1} \quad \forall S \in \calS_k(\calM)\]
    where $L=\max \{\norm{A} \setsep A \in \calM\}$.
  \end{description}
\end{con}

\begin{figure}
\begin{flushleft}
  \begin{tabular}{
    C{0.41\linewidth}
    C{0.04\linewidth}%
    C{0.41\linewidth}%
    }
\boxed{%
\begin{minipage}[m]{0.99\linewidth}
      \begin{center}
local\\ 
$\alpha$-Hölder continuity
\end{center}
\vspace*{-0.3cm}
\hspace*{-1mm}\begin{tabular}{C{0.44\linewidth}|%
C{0.44\linewidth}}
\hline\\[-4mm]
\Cref{cor:local-hoelder-dim-2} & Conj. L1 \\
on $\calH_{>0}(2)$ & on $\calH(d)$ \\
$ \alpha = \frac{1}{d^2+d}$ & $\alpha = \frac{1}{d}$
\end{tabular}
\end{minipage}
}
&  
$\stackrel{~}{\Longrightarrow}$
&
\boxed{\begin{minipage}[m]{0.99\linewidth}
        \begin{center}
pointwise\\ $\alpha$-Hölder continuity
        \end{center}
\vspace*{-0.3cm}
\hspace*{-1mm}\begin{tabular}{C{0.44\linewidth}|%
C{0.44\linewidth}}
\hline\\[-4mm]
\Cref{thm:pointwise-Hoelder-everywhere-v1} & Conj. P1 \\
$ \alpha = \frac{1}{d^2+d}$ & $\alpha = \frac{1}{d}$
\end{tabular}
  \end{minipage}
}
  \\[4ex]
~\vspace*{0.3cm}
  
\begin{minipage}[m]{0.99\linewidth}
\begin{tabular}{C{0.42\linewidth}%
C{0.04\linewidth}%
C{0.42\linewidth}}
\Cref{lem:independent-lower-bound-implies-locally-hoelder} & $\Uparrow$ &
\end{tabular}
\end{minipage}

& & 
~\vspace*{0.3cm}

\begin{minipage}[m]{0.99\linewidth}
\begin{tabular}{C{0.42\linewidth}%
C{0.04\linewidth}%
C{0.42\linewidth}}
\Cref{thm:pointwiseHoelder-v1}
& $\Uparrow$ & 
\end{tabular}
\end{minipage}
\\[2ex]
%%%% LEVEL 2
\boxed{\begin{minipage}[m]{0.99\linewidth}
      \begin{center}
local\\ 
$\alpha$-Hölder lower bounds
\end{center}
\vspace*{-0.3cm}
\hspace*{-1mm}\begin{tabular}{C{0.44\linewidth}|%
C{0.44\linewidth}}
\hline\\[-4mm]
Proof of Corollary \ref{cor:local-hoelder-dim-2} & Conj. L2\\
on $\calH_{>0}(2)$ & on $\calH(d)$ \\
$ \alpha = \frac{1}{d^2+d}$ & $\alpha = \frac{1}{d}$
\end{tabular}
\end{minipage}
}
&  
&
\boxed{\begin{minipage}[m]{0.99\linewidth}
        \begin{center}
pointwise\\ $\alpha$-Hölder 
lower bounds
        \end{center}
\vspace*{-0.3cm}
\hspace*{-1mm}\begin{tabular}{C{0.44\linewidth}|%
C{0.44\linewidth}}
\hline\\[-4mm]
  & Conj. P2 \\
$ \alpha = \frac{1}{d^2+d}$ & $\alpha = 1$
\end{tabular}
  \end{minipage}
}
  \\[5ex]

~\vspace*{0.3cm}
  
\begin{minipage}[m]{0.99\linewidth}
\begin{tabular}{C{0.42\linewidth}%
C{0.04\linewidth}%
C{0.42\linewidth}}
\Cref{lem:independent-product-bound-implies-independent-lower-bound}: & $\Uparrow$ &
 \hspace*{-0.5cm}
 \begin{minipage}[l]{0.8\linewidth}
 on $\calH_{>0}(d)$\\ 
 $\alpha= \frac{1}{d^2+d}$
 \end{minipage}
\end{tabular}
\end{minipage}

& & 
~\vspace*{0.3cm}

\begin{minipage}[m]{0.99\linewidth}
\begin{tabular}{C{0.42\linewidth}%
C{0.04\linewidth}%
C{0.42\linewidth}}
\Cref{thm:hoelder-lower-bound-allr}: & $\Uparrow$ &
$\alpha= \frac{1}{d^2+d}$\end{tabular}
\end{minipage}
\\[2ex]
%%%% LEVEL 3
\boxed{\begin{minipage}[m]{0.99\linewidth}
      \begin{center}
locally uniform\\ 
exponential-polynomial\\
trajectory bounds on $\calH_{>0}(d)$
\end{center}
\vspace*{-0.3cm}
\hspace*{-1mm}\begin{tabular}{C{0.44\linewidth}|%
C{0.44\linewidth}}
\hline\\[-4mm]
\Cref{lem:polynomial-growth-bound-plane} & Conj. L3 \\
$ d= 2$ & any $d \in\N$
\end{tabular}
\end{minipage}
}
&  
$\stackrel{~}{\Longrightarrow}$
&
\boxed{\begin{minipage}[m]{0.99\linewidth}
        \begin{center}
exponential-polynomial\\ 
trajectory bounds
        \end{center}
        \begin{center}
\textbf{Well-known}, Lemma~\ref{lem:polynomial-growth-bound}
        \end{center}
  \end{minipage}
}
\end{tabular}
\end{flushleft}
    \caption{Summary of properties proved in the paper versus those conjectured by the authors. In the left of a box the proved results are stated, whereas in the right the respective conjecture is shown. The statements labeling the implication arrows between boxes concern statements that show that the lower level statements (results as well as conjectures) imply the upper level.}
    \label{fig:overview}
\end{figure}

In view of condition L3, we also introduce the notation
\begin{equation*}
    \calH_{>0}(d) := \{ \calM\in \calH(d) \setsep \rho(\calM) > 0 \}
\end{equation*}
and note that $\calH_I(d) \subseteq \calH_{>0}(d) \subseteq \calH(d)$.

The reader will recognize that the Conjectures L1 and P2 are simply
versions of Elsner's theorem, Theorem~\ref{thm:elsner}, respectively
Proposition~\ref{prop:specrad-lowerbound} generalized to the context
of the joint spectral radius. The role of L3 is maybe less
obvious. In this paper it is shown that we have the
implication
\begin{equation}
\begin{aligned}
    L3 \quad &\Rightarrow \quad L2 \text{ on } \calH_{>0}(d) \text{ with } \alpha = \frac{1}{d^2 +d} \\ &\Rightarrow \quad L1 \text{ on } \calH_{>0}(d) \text{ with } \alpha = \frac{1}{d^2 +d}.    
\end{aligned}
\end{equation}
In other words it would be sufficient to prove Conjecture~L3 to
finalize the proof that the joint spectral radius is indeed locally
Hölder continuous, albeit not with the expected optimal constant
$\alpha = 1/d$ and only on $\calH_{>0}(d)$. To us it was somewhat of a
surprise that such a seemingly innocent statement has proved to be so
intransigent to our attempts of resolving the question.

\section{Preliminaries}
\label{sec:preliminaries}

A helpful formulation of the joint spectral radius is in terms of operator norms. Given a norm $\norm{\cdot}$ on $\K^d$ and its induced operator norm also denoted by $\norm{\cdot}$, we define
\begin{equation*}
    \norm{\mathcal{M}} := \max \{\norm{A} \setsep A \in \mathcal{M}\}.
\end{equation*}
Then it is known, \cite[Proposition 1]{RotaStra60}, that
\begin{equation}
    \rho(\mathcal{M)} = \inf \{ \norm{\mathcal{M}} \setsep \norm{\cdot} \text{ is an operator norm} \}.
\end{equation}
An operator norm $\norm{\cdot}$ is called extremal for $\mathcal{M}$, if $\rho(\mathcal{M}) = \norm{\mathcal{M}}$. An extremal norm is called a Barabanov norm, if in addition for every $x\in \K^d$ there exists an $A \in \mathcal{M}$ such that
\begin{equation*}
    \norm{Ax} = \rho(\mathcal{M}) \norm{x}. 
\end{equation*}
A sufficient condition for the existence of Barabanov norms is that the set $\mathcal{M}$ is irreducible, i.e. only the trivial subspaces $\{0\}$ and $\K^d$ are invariant under all $A\in \mathcal{M}$, \cite{Bara88,Wirt02}.

\subsection{Elementary inequalities}

Here we collect some known inequalities that will be helpful in the sequel.
The following proposition can be derived with various
constants from results of Bochi~\cite{bochi2003inequalities}, Morris~\cite{morris2010rapidly} and  Breuillard~\cite{breuillard2022joint}.

\begin{prop}
  \label{prop:converging-lower-bound}
  There is a constant $\Lambda_d$ depending only on the dimension $d$ 
  such that  for any bounded set of matrices  $\calM$ in $\C^{d \times d}$
  we have
  \begin{align*}
    \max_{1\leq k \leq n} \sup_{S \in \calS_k(\calM)} \rho(S)^{1/k} \geq \rho(\calM)(1- \frac{\Lambda_d}{n} ). 
  \end{align*}  
\end{prop}
The statement of \Cref{prop:converging-lower-bound} is basically contained in the introduction of
  Morris \cite{morris2010rapidly}, the displayed formula after Theorem 1.2 therein, but we need slightly more explicit estimates. 
\begin{proof}
  By Bochi's inequality, \cite[Theorem B]{bochi2003inequalities}, there are 
  an integer $m_d$ and a constant $C_d \geq 1$, both only depending on $d$,
  such that
  \begin{align}
    \label{eq:bochi}
    \rho(\calM) \leq C_d\max_{1\leq k \leq m_d} \sup_{S \in \calS_k(\calM)} \rho(S)^{1/k} .
  \end{align}
  Let $n \geq m_d$ and let $\ell \in \N$ with
  $\ell m_d \leq n < (\ell+1) m_d \leq 2\ell m_d$.
  Then
  \begin{align*}
  \rho(\calM)=\rho(\calS_\ell(\calM))^{1/\ell} 
  &\leq C_d^{1/\ell} \max_{1\leq k \leq m_d} \sup_{S \in \calS_{k\ell}(\calM)} \rho(S)^{1/(k\ell)} \\
  &\leq C_d^{1/\ell} \max_{1 \leq k \leq \ell m_d} \sup_{S \in \calS_k(\calM)} \rho(S)^{1/k}\\
  &\leq C_d^{1/\ell} \max_{1 \leq k \leq n} \sup_{S \in \calS_k(\calM)} \rho(S)^{1/k}\\
  &\leq C_d^{2m_d/n} \max_{1 \leq k \leq n} \sup_{S \in \calS_k(\calM)} \rho(S)^{1/k}.
  \end{align*}
  Hence 
  \begin{align*}
    \max_{1\leq k \leq n} \sup_{S \in \calS_k(\calM)} \rho(S)^{1/k}
    \geq \rho(\calM) (C_d^{2m_d})^{-1/n}
  \end{align*}
  for $n \geq m_d$. For $n <m_d$ the same follows
  directly from \eqref{eq:bochi}.
  The result follows with $\Lambda_d:=C_d^{2m_d}$ from
  the validity of
  $1-\frac{x}{n} \leq x^{-1/n}$
  for all $x,n \geq 1$.
  \end{proof}

  For state of the art estimates on $C_d$ and $m_d$ see \cite{breuillard2022joint}.

The following lemma is mentioned in
\cite{varneyMarginalGrowthRates2022}
and is also an immediate consequence of  
\cite[Theorem 3.1]{guglielmiAsymptoticPropertiesFamily2001}.
The idea is to use the fact 
(due to Barabanov)
that a family of matrices with
joint spectral radius equal to $1$ and unbounded growth must be reducible.
Unfortunately, it is not clear that the constant obtained in that way
depends continuously on the matrix family, see Conjecture L3.

\begin{lem}
  \label{lem:polynomial-growth-bound}
  Let $\calM \in \calH(d)$ be a compact set of matrices with $\rho(\calM)>0$.
  There is a constant $\Theta>0$ depending on $\calM$ such that
  \begin{equation}
  \label{eq:lemma5-bound}
    \norm{A_k \dots A_1} \leq \Theta k^{d-1}\rho(\calM)^{k}
  \end{equation}
  for all $k\in \N$, $A_i \in \calM$, $i=1,\ldots,k$.
\end{lem}

\begin{rem}
  \label{rem:L3-altenative-formulation}
  It is not difficult to show that Conjecture L3 is equivalent to the following:
  For every
    compact set $\calK \subseteq \mathcal{H}(d)$ there is a constant
    $\Theta$ such that for every 
    $\mathcal{M}\in \calK$
    it holds
    that
    \[\norm{A_k \dots A_1} \leq\begin{cases}
                 \Theta k^{d-1} \rho(\calM)^k \text{ for } k\geq
                     1&\text{ if } \rho(\calM)\geq 1,\\
                 \Theta k^{d-1} \rho(\calM)^{k-(d-1)} \text{ for }
                   k\geq d-1 &\text{ if } \rho(\calM) \in [0,1)
               \end{cases}
             \]
    for every $A_i \in \calM$, $i=1,\ldots,k$.
In other words, we conjecture that \Cref{lem:polynomial-growth-bound}
holds with a uniform constant on $\calK$.
\end{rem}

The following example shows that one cannot simply require the first of the previous two inequalities for all $\calM$ with $\rho(\calM)>0$.

 \begin{example}
     \label{example:L3isintrouble}
     Consider the set $\calK \subseteq \calH(2)$ consisting of singletons given as follows.
     \begin{equation*}
         \calK = \{ \calM(\varepsilon) \setsep \varepsilon \in [0,1] \},\ \calM(\varepsilon) := \{ A_\varepsilon\}, \ A_\varepsilon:= \begin{bmatrix}
             \varepsilon & 1 \\ 0 & \varepsilon
         \end{bmatrix}, \varepsilon \in [0,1].
     \end{equation*}
     It is clear that $\rho(\calM(\varepsilon)) = \varepsilon$ and $\norm{A_\varepsilon}_2 \geq 1$. Thus for $k=1$ it is impossible to satisfy 
     the first inequality in Remark~\ref{rem:L3-altenative-formulation} with a uniform constant $\Theta$. Moreover, for $k> 1$ we have
     \begin{equation*}
         A_\varepsilon^k = \begin{bmatrix}
             \varepsilon^k & k \varepsilon^{k-1} \\
             0 & \varepsilon^k
         \end{bmatrix}
     \end{equation*}
     and the condition
     \begin{equation*}
         k \varepsilon^{k-1} \leq \norm{A_\varepsilon^k}_2 
         \leq \Theta \varepsilon^k k
     \end{equation*}
     leads to the impossible condition $1 \leq \Theta \varepsilon$, $\varepsilon\in (0,1]$. Thus even a relaxation of the first condition to just requiring it for larger $k$ would not resolve the issue. The second condition on the other hand is easily satisfied.
 \end{example}

\subsection{Hölder continuity}
\label{subsec:Hölder}
Here we recall a few definitions connected to Hölder continuity. 
We use these notions in the way they are defined e.g.
in \cite{bauerProbabilityTheory2011}
or \cite{bezerraUpperBoundRegularity2023}.
\begin{defn}
\label{def:Hölder}
Let $(X,d_X),(Z,d_Z)$ be metric spaces. A function $f:X \to Z$ is called 
\begin{enumerate}[(i)]
    \item \emph{Hölder continuous at $x_0\in X$} with exponent $\alpha$, if
    \begin{equation}
        \limsup_{x\to x_0} \frac{d_Z(f(x), f(x_0))}{d_X(x,x_0)^{\alpha}} < \infty.
    \end{equation}
    We call $f$ \emph{pointwise Hölder continuous}, if it is Hölder continuous at every $x_0\in X$.
    \item \emph{Hölder continuous on $X$} with exponent $\alpha$, if
    \begin{equation}
        \sup_{x,y\in X, x\neq y} \frac{d_Z(f(x), f(y))}{d_X(x,y)^{\alpha}} < \infty.
    \end{equation}
    \item  \emph{locally Hölder continuous on $X$} with exponent $\alpha$, if for every compact subset $K\subseteq X$ it is Hölder continuous on $K$ with exponent $\alpha$.
\end{enumerate}
\end{defn}

Note that Hölder continuity at a point is a local property.
Clearly, if $X$ is compact, then local Hölder continuity implies Hölder continuity, while this may fail if $X$ is not compact. However,
even for compact $X$,  pointwise Hölder continuity does not imply
Hölder continuity. Regarding Hölder continuity, we require the following simple observations.

\begin{lem}
  \label{lem:Hoelder-via-integers}
  Let $p,q \in \N$ and let $(X,d_X)$ be a metric space.  Assume
  that for $f: X \to \R$ and $x_0 \in X$ there are an integer $n_0 >0$
  and constants $\tau, \Omega>0$ such that for all $y \in X$ and all
  $n\geq n_0$
  \begin{align*}
    d_X(x_0,y) \leq  \frac{\tau}{n^q} \quad\implies \quad f(x_0) \geq f(y)-\frac{\Omega}{n^p}.
  \end{align*}
  Then
  \begin{align*}
    f(x_0) \geq f(y) - (2^p \Omega \tau^{-p/q}) d_X(x_0,y)^{p/q}
  \end{align*}
  for all $y \in X$ with $d_X(x_0,y) \leq \tau n_0^{-q}$.
\end{lem}
\begin{proof}
  Let $0<d_X(x_0,y)\leq \tau n_0^{-q}$.
  There is
  $n \in \N, n \geq n_0$
  such that
  \[\frac{\tau}{(2n)^q} \leq \frac{\tau}{(n+1)^q} < d_X(x_0,y) \leq \frac{\tau}{n^q}.\]
  Hence
  \[f(x_0) \geq f(y)-\Omega n^{-p}
  = f(y)-2^p \Omega ((2n)^{-q})^{p/q}
  \geq f(y)-2^p \Omega \tau^{-p/q}
  d_X(x_0,y)^{p/q}.\qedhere\]
\end{proof}

\begin{lem}
    \label{lem:Hoeldersimple1}
    Let $f:[0,\eta] \to \R$ be a continuous function that is
    $\alpha$-Hölder at $0$. Then there exists a constant $C_\eta$ such that
    \begin{equation}
        \label{eq:Hoeldersimple1}
        \abs{f(x) - f(0)} \leq C_\eta x^{\alpha}, \quad x \in [0,\eta].
    \end{equation}
\end{lem}

\begin{proof}
    By Definition~\ref{def:Hölder} (i) there exists an $\varepsilon_0>0$ and a $\tilde{C}>0$ such that $\abs{f(x) - f(0)} \leq \Tilde{C} x^\alpha$ for all $x \in [0,\varepsilon_0]$.
    If $\varepsilon_0=\eta$ we are done; otherwise let $h:= \max \{ \abs{f(x)-f(0)} \setsep x \in [\varepsilon_0,\eta]\}$. Set $C_\eta:= \max\{ \tilde{C}, h \varepsilon_0^{-\alpha}\}$, then \eqref{eq:Hoeldersimple1} follows for all $x\in[0,\eta]$.
\end{proof}

\section{Reducibility and norms}
\label{sec:norms}

In this section we recall the property of reducibility and briefly discuss particular norms that are adapted to the reducibility structure of $\calM\in\calH(d)$. Recall that $\calM \in \calH(d)$ is called reducible, if there is a nontrivial subspace $X \subseteq \K^{d}$ such that $X$ is $A$-invariant for all $A\in\calM$. For a reducible $\calM$ there is a maximal flag $\{0\} = F_0 \subsetneq F_1 \subsetneq \cdots \subsetneq F_m = \K^{d}$ such that every subspace $F_j$, $j=0,\ldots,m$, is $A$-invariant for all $A\in\calM$. We will call the length $m$ of this maximal flag the reducibility index of $\calM$. Note that with this convention a reducibility index of $1$ means that the system is in fact irreducible, i.e. no nontrivial joint invariant subspace exists.

We need the following property of norms defined with respect to a flag.
\begin{lem}
\label{lem:normbound}
Let $\K^d = \bigoplus_{i=1}^m X_i$ for subspaces $X_i, i=1,\ldots,m$. Assume each $X_i$ is endowed with a norm $v_i$. Let $\pi_i: \K^d\to X_i$ be the projection from $\K^d$ to $X_i$ along the complementary subspace $\bigoplus_{j\neq i} X_j$ and let
 $\imath_i:X_i\to\K^d$ be the canonical injection.
 Define a norm $v$ on $\K^d$ by
\begin{equation}
\label{eq:vdef}
    v(x) = \norm{ (v_1 (\pi_1(x)), \ldots, v_m(\pi_m(x)))}_2. 
\end{equation}
Then we have for the operator norm induced by $v$ and any $A\in \K^{d\times d}$ that
\begin{equation}
\label{eq:nicenorm-euclidean}
    v(A) \leq \norm{ (r_{ij})_{i,j=1}^m}_2,
\end{equation}
where we denote by $v_{ij}$ the induced operator norm from $(X_j, v_j)$ to $(X_i,v_i)$ and
\begin{equation}
    r_{ij} := v_{ij}(\pi_i A \imath_j).
\end{equation}
\end{lem}

\begin{proof}
Let $x=\sum_{i=1}^m x_i \in \K^d$ with $v(x)=1$ and $x_i\in X_i$, $i=1,\ldots,m$. Then
\begin{multline*}
    v(Ax) = \norm{\begin{bmatrix}
    v_1(\sum_{i=1}^m \pi_1 A  x_i) \\
    \vdots \\
    v_m(\sum_{i=1}^m \pi_m A  x_i)
    \end{bmatrix}}_2
    \leq
    \norm{\begin{bmatrix}
    \sum_{i=1}^m v_1(\pi_1 A \imath_i x_i) \\
    \vdots \\
    \sum_{i=1}^m v_m(\pi_m A \imath_i x_i)
    \end{bmatrix}}_2 \\
    \leq \norm{\begin{bmatrix}
    \sum_{i=1}^m v_{1i}(\pi_1 A \imath_i)v_i(x_i) \\
    \vdots \\
    \sum_{i=1}^m v_{mi}(\pi_m A \imath_i)v_i(x_i)
    \end{bmatrix}}_2 \\
    = \norm{\begin{bmatrix}
    v_{11}(\pi_1 A \imath_1) &\ldots & v_{1m}(\pi_1 A \imath_m) \\
    \vdots&& \vdots \\
    v_{m1}(\pi_m A \imath_1) &\ldots & v_{mm}(\pi_m A \imath_m)
    \end{bmatrix} 
    \begin{bmatrix}
    v_1(x_1) \\ \vdots \\ v_m(x_m)
    \end{bmatrix}}_2
\end{multline*}
and the claim follows as we have assumed that $x$ is arbitrary with $v(x) =1$.
\end{proof}

Let $v$ be a norm on $\K^d$. Then we define
\begin{eqnarray}
    \label{eq:112}
    c^-(v) &:=& \min \{ v(x) \setsep \norm{x}_2 = 1 \} \,,\\ 
    c^+(v) &:=& \max \{ v(x) \setsep \norm{x}_2 = 1 \} \,. 
\end{eqnarray}
Note that for any $A\in \K^{d \times d}
$ we have for the induced operator norms that
\[ \frac{c^-(v)}{c^+(v)}\norm{A}_2 \le v(A) \leq
\frac{c^+(v)}{c^-(v)}\norm{A}_2 \,.\]  
We define the eccentricity of a norm $v$ with respect to the Euclidean norm by 
\begin{equation}
    \label{eq:def-ecc}
    \ecc_2(v) := \frac{c^+(v)}{c^-(v)}.
\end{equation}

The following observation on the norms constructed in \Cref{lem:normbound} will be helpful in estimating constants in Hölder estimates.

\begin{lem}
\label{lem:normbound-eccentricity}
Let $\K^d = \bigoplus_{i=1}^m X_i$ for pairwise orthogonal subspaces $X_i, i=1,\ldots,m$. Assume each $X_i$ is endowed with a norm $v_i$ satisfying $c^-(v_i)=1$. Let $\pi_i: \K^d\to X_i$ be the orthogonal projection from $\K^d$ to $X_i$ 
and define a norm $v$ on $\K^d$ by
\begin{equation}
    v(x) = \norm{ (v_1 (\pi_1(x)), \ldots, v_m(\pi_m(x)))}_2. 
\end{equation}
Then we have for the operator norm induced by $v$ that
\begin{equation}
\ecc_2(v) = \max_{i=1,\ldots,m} \ecc_2(v_i).
\end{equation}
\end{lem}

\begin{proof}
For $i=1,\ldots,m$ we have $c^-(v_i)=1$ and so $v_i(x) \geq \norm{x}_2$ for $x\in X_i$. By pairwise orthogonality of the $X_i$, it follows for $x \in \K^d$  that
\begin{multline*}
    v(x) = \norm{ (v_1 (\pi_1(x)), \ldots, v_m(\pi_m(x)))}_2 \\
    \ge \norm{(\norm{\pi_1(x)}_2, \ldots, \norm{\pi_m(x)}_2)}_2 = \norm{x}_2.
\end{multline*}
This implies $c^-(v) = 1$ as the minimum is attained in each $X_i$.

For each $i=1,\ldots, m$ we have
\begin{equation*}
 \ecc_2(v) = c^+(v) \geq \max \{ v(x) \setsep \norm{x}_2 = 1, x \in X_i \} = \ecc_2(v_i).
\end{equation*}
So that $\ecc_2(v) \geq \max_{i=1,\ldots,m} \ecc_2(v_i)$. 
To prove the converse, note that for each $x_i \in X_i$ we have
\begin{equation}
\label{eq:ccomp}
     v_i(x_i) \leq c^+(v_i) \norm{x_i}_2.
\end{equation}
Using $c^-(v) = 1$ again, it follows that 
\begin{align*}
    \ecc_2(v) &= c^+(v)  = \max_{\norm{x}_2=1}\norm{(v_1(x_1),\ldots,v_m(x_m))}_2\\
    &\leq \max_{\norm{x}_2=1}\norm{(c^+(v_1)\norm{x_1}_2,\ldots,c^+(v_m)\norm{x_m}_2)}_2 \\
    &\leq \max_{i=1,\ldots,m}c^+(v_i) \max_{\norm{x}_2=1}\norm{(\norm{x_1}_2,\ldots,\norm{x_m}_2)}_2 \\
    &= \max_{i=1,\ldots,m}c^+(v_i) = \max_{i=1,\ldots,m}\ecc(v_i).
\end{align*}
This shows the assertion.
\end{proof}

We note that without a normalizing constraint like the one chosen in Lemma~\ref{lem:normbound-eccentricity}, i.e. that $c^-(v_i)=1$ for all $i$, the claim of the lemma cannot be true. For instance, if $\K^d = X_1 \bigoplus X_2$ with corresponding norms $v_1,v_2$, the eccentricity of $v_\alpha(x) = \norm{(v_1(x_1), \alpha v_2(x_2)}$, $\alpha >0$, can be arbitrarily large.

\section{$\varepsilon$-inflation}
\label{sec:increasing}

In this section we consider the change of the joint spectral radius as norm balls in subspaces are added to an original compact set $\mathcal{M}$.
To set the stage, 
let $V$ be a linear subspace of $\K^{d\times d}$ and let $\norm{\cdot}$ be a norm on $\K^{d\times d}$ (not necessarily a matrix norm).
Define the unit ball in $V$ by
\begin{equation}
    \label{eq:def:BV}
    \calB_V := \{ A \in V \setsep \norm{A} \leq 1 \}.
\end{equation}
Let $\mathcal{M}\in \mathcal{H}(d)$. We consider maps of the type
\begin{equation}
\label{eq:def:Meps}
    \varepsilon \mapsto \mathcal{M}_\varepsilon := \mathcal{M} + \varepsilon \calB_V.
\end{equation}
and want to investigate the regularity of the map $r: \varepsilon \mapsto \rho(\mathcal{M}_\varepsilon)$. We first show a Hölder continuity property of this map at $\varepsilon=0$ and in the next step extend this to Hölder continuity on compact intervals. The Hölder exponent turns out to depend on the index of reducibility of $\mathcal{M}$.

\begin{prop}
\label{prop:hoel1}
Let $\norm{\cdot}$ be a norm on $\K^d$ and
let $\mathcal{M}\in\mathcal{H}(d)$ have index of reducibility $m$. Let $V$ be a subspace of $\K^{d\times d}$ and define $\mathcal{M_\varepsilon}:=\mathcal{M} + \varepsilon \calB_V$, $\varepsilon>0$.
Then 
\begin{equation*}
    r : \varepsilon\mapsto \rho\left(\mathcal{M}_\varepsilon\right)
\end{equation*}
is increasing. In addition, it is Hölder continuous at $0$ with exponent $1/m$. In particular, for any $\eta>0$ there exists a constant $C_\eta$ such that
\begin{equation}
    r(\varepsilon) - r(0) \leq C_\eta \varepsilon^{1/m}, \quad \varepsilon \in [0,\eta].
\end{equation}
\end{prop}

\begin{proof} 
It is clear that $r$ is increasing.
For reducibility index $m=1$ we are in the irreducible case and the result follows from \cite{Wirt02}.

Let $\mathcal{M}$ have reducibility index $m \geq 2$. 
It is sufficient to prove the result for one specific norm. 
Given two norms $v_1$ and $v_2$ on $\K^{d\times d}$, we use the upper index $v_i$ to denote the objects defined in \eqref{eq:def:BV}, \eqref{eq:def:Meps} depending on the norm $v_i$.
By the equivalence of norms, for given norms $v_1$ and $v_2$ there is a constant $D$ such that $\varepsilon \calB_V^{v_2}\subseteq D\varepsilon \calB^{v_1}_{V}$ and so $\rho(\mathcal{M}^{v_2}_\varepsilon) \leq \rho(\mathcal{M}^{v_1}_{D\varepsilon})$.
Thus if the result is known for $v_1$ with a constant $C_\eta$ on the interval $\eta$, then we have on $[0,\eta/D]$ that
\begin{equation}
\label{eq:D-bound-in-prop-hoel1}
    \rho(\mathcal{M}^{v_2}_\varepsilon) - \rho(\mathcal{M}) \leq \rho(\mathcal{M}^{v_1}_{D\varepsilon}) - \rho(\mathcal{M}) \leq C_\eta D^{1/m} \varepsilon^{1/m}.
\end{equation}

We thus begin by fixing a suitable norm.
Let $(F_0,\ldots,F_m)$ be a flag corresponding to the reducibility index of $\mathcal{M}$ and choose pairwise orthogonal spaces $X_1,\ldots,X_m$ such that
\begin{equation}
    F_{i-1} \oplus X_i = F_{i}, \quad i=1,\ldots,m.
\end{equation}
As before let $\pi_i$ be the orthogonal projection $\pi_i:\K^d\to X_i$ and $\imath_i:X_i\to\K^d$ the canonical injection. We define
\begin{equation}
    \calM_{ij} := \left\{ \pi_i A \imath_j \setsep A \in \calM \right\}, \quad 1 \leq i,j \leq m.
\end{equation}

The sets $\mathcal{M}_{ii}$ of the restrictions of $\pi_i A$ to $X_i$, $A\in \mathcal{M}$, are irreducible or equal to $\{0\}$ and so, \cite{Bara88,Wirt02}, we may choose Barabanov norms $v_i$ for $\mathcal{M}_{ii}$, $i=1,\ldots,m$. In particular, we have (see also \cite[Lemma 2]{BergWang92})
\begin{equation}
    \rho(\mathcal{M}) = \max_{i=1,\ldots,m}\rho(\mathcal{M}_{ii}) = \max_{i=1,\ldots,m} v_i(\mathcal{M}_{ii}).
\end{equation}
We now define the norm $v$ as in \eqref{eq:vdef}. 
Again denote by $v_{ij}$ the induced operator norm from $(X_j, v_j)$ to $(X_i,v_i)$.
By similarity scaling we may assume without loss of generality that
\begin{equation*}
    v_{ij}(\mathcal{M}_{ij}) \leq 1, \quad 1 \leq i< j \leq m.
\end{equation*}
On the other hand of course $\mathcal{M}_{ij} =\{0\}$ for $i>j$.

It suffices to show that there exist $\varepsilon_0>0$ and $C>0$ such that for all $\varepsilon \in (0,\varepsilon_0]$ we have 
\begin{equation*}
    \rho(\mathcal{M}_\varepsilon) - \rho(\mathcal{M}) \leq C \varepsilon^{1/m}.
\end{equation*}
The full claim then follows from Lemma~\ref{lem:Hoeldersimple1}.

Using the properties of $v$, for $\varepsilon >0$ we then have for the blocks of the matrices $A\in \mathcal{M}_\varepsilon$ that
\begin{equation}
    v_{ij}(\mathcal{M}_{\varepsilon,ij}) \leq \left\{
    \begin{matrix}
    \rho(\mathcal{M}_{ii})+ \varepsilon& \quad& \text{ for } i=j, \\
    1+\varepsilon& \quad& \text{ for } i<j, \\
    \varepsilon& \quad& \text{ for } i>j.
    \end{matrix}\right.
\end{equation}
As the joint spectral radius is invariant under similarity transformation, we can now rescale via a diagonal transformation of the form
\begin{equation*}
    T_\varepsilon = \diag ( I_1,  \delta I_2, \ldots, \delta^{m-1} I_m),
\end{equation*}
where $I_j$ is the identity matrix of dimension $\dim X_j$, $j=1,\ldots,m$ and
\begin{equation}
    \delta = \sqrt[m]{\varepsilon}.
\end{equation}

Using Lemma~\ref{lem:normbound} we obtain for any matrix in 
$A \in T_\varepsilon^{-1} \mathcal{M}_\varepsilon T_\varepsilon$ that
\begin{equation}
    v(A) \leq \norm{
    \begin{bmatrix}
    \rho(\mathcal{M}_{11}) + \varepsilon & (1+\varepsilon) \varepsilon^{1/m}& \ldots  & (1+\varepsilon)\varepsilon^{(m-1)/m} \\
    \varepsilon^{(m-1)/m} &&&\\
    \vdots &&&\\
     \varepsilon^{1/m} &&&  \rho(\mathcal{M}_{mm}) + \varepsilon
    \end{bmatrix}
    }_2.
\end{equation}
Denoting the matrix on the right by $Q(\varepsilon) = ( q_{ij}(\varepsilon))_{i,j=1}^m$, we see that the diagonal entries of $Q(\varepsilon)$ are of the form
\begin{equation*}
    q_{jj}(\varepsilon) =  \rho(\mathcal{M}_{jj}) + \varepsilon, \quad j=1,\ldots,m,
\end{equation*}
while we have for the off-diagonal entries 
that
\begin{equation*}
    0\leq q_{ij}(\varepsilon) \leq 2\varepsilon^{1/m}
    , \quad i\neq j,
\end{equation*}
for $\varepsilon \in [0,1]$.
We also note that $\rho(\mathcal{M}) = \max_{i=1,\ldots,m} \rho(\mathcal{M}_{ii}) = \norm{Q(0)}_2$.
Consequently, we have
\begin{align*}
    \rho(\mathcal{M}_\varepsilon) - \rho(\mathcal{M})
    &=\rho(T_\varepsilon^{-1}\mathcal{M}_\varepsilon T_\varepsilon) - \rho(\mathcal{M}) \\
    &\leq
    v(T_\varepsilon^{-1}\mathcal{M}_\varepsilon T_\varepsilon) - \rho(\mathcal{M})\\
    &\leq \norm{Q(\varepsilon)}_2 - \norm{Q(0)}_2 \leq
    \norm{ Q(\varepsilon) - Q(0)}_2 \leq C \varepsilon^{1/m}
\end{align*}
for all $0\leq \varepsilon \leq 1$ and a suitable positive constant $C$.
\end{proof}

\begin{rem}
  The constant $C_\eta$ asserted by \Cref{prop:hoel1} can be determined more accurately with the help of \Cref{lem:normbound-eccentricity}. On the one hand the choice of the constant $C$ bounding $\norm{Q(\varepsilon)-Q(0)}_2$ in the proof is quite generic and only marginally involves the problem data. The extension to $[0,\eta]$ by \Cref{lem:Hoeldersimple1} is also straightforward. On the other hand, we have in \eqref{eq:D-bound-in-prop-hoel1} that $D=\ecc(v)$ where $\ecc(v)$ is the eccentricity of the chosen norm $v$ with respect to our original norm of interest. If the latter happens to be the spectral norm, then \Cref{lem:normbound-eccentricity} gives a fairly good understanding of what $D$ is. In this case the results of \cite{KOZYAKIN201012} can be used to obtain easier estimates for the constants involved.
\end{rem}

To extend \Cref{prop:hoel1} to a statement of local Hölder continuity, we need a convexity argument. To this end,
recall that the joint spectral radius is invariant under the operations of taking the convex hull and balancing, see \cite{Bara88,Jungers}. We denote the convex hull of a set $Z\subseteq \K^{d\times d}$ by $\mathrm{conv}(Z)$.
\begin{lem}[\cite{Bara88}]
\label{lem:conv-balan}
Let $\mathcal{M}\subseteq \K^{d\times d}$ be compact. Then
\begin{equation*}
\rho \mathcal{(M)}=\rho\left(\mathrm{conv}\left(\mathcal{M}\cup-\mathcal{M}\right)\right)\,.
\end{equation*}
\end{lem}

With this we obtain the main result of this section.

\begin{thm}
\label{theo:Hölder}
Let $\mathcal{M}\subseteq \K^{d\times d}$ be compact and of reducibility index $m\geq 1$. Consider the set $\calB_V$ defined in \eqref{eq:def:BV} and the increasing map in \eqref{eq:def:Meps}
\begin{equation*}
\varepsilon \mapsto \mathcal{M}_{\varepsilon}:=\mathcal{M}+\varepsilon \calB_{V}, \quad \varepsilon\geq 0\,.
\end{equation*}
Then the map $r:[0,\infty)\rightarrow \R$
\begin{equation*}
r(\varepsilon):=\rho(\mathcal{M}_{\varepsilon})
\end{equation*}
is locally $\frac{1}{m+1}$-Hölder continuous, if $m\geq 2$, and locally Lipschitz continuous if $m=1$.
\end{thm}

\begin{proof}
The irreducible case $m=1$ follows from \cite{Wirt02}. Let $m\geq 2$.

Fix $\eta >0$. We will show that $r$ is $\frac{1}{m+1}$-Hölder on $[0,\eta]$.
By Proposition~\ref{prop:hoel1} we know that $r$ is $\frac{1}{m}$-Hölder at $\varepsilon=0$. 

By Lemma~\ref{lem:conv-balan} we may assume that $\mathcal{M}$ is convex and balanced, because
\begin{multline*}
 r(\varepsilon) = \rho(\calM + \varepsilon \calB_V) = \rho(\mathrm{conv}((\mathcal{M} + \varepsilon \calB_V) \cup - (\mathcal{M} + \varepsilon \calB_V))) 
\\
= \rho(\mathrm{conv}((\mathcal{M}  \cup - \mathcal{M}) + \varepsilon \calB_V)) = 
\rho(\mathrm{conv}(\mathcal{M}  \cup - \mathcal{M}) + \varepsilon \calB_V).
\end{multline*}
This implies, in particular, that $\mathcal{M}_{\varepsilon}$ is convex and balanced for all $\varepsilon\geq 0$, as $\calB_{V}$ is convex and balanced.

Fix $\varepsilon >0, \delta >0$, so that $\varepsilon+\delta \leq \eta$. We claim that
\begin{equation}
\label{eq:eps_delta1}
0\leq r(\varepsilon+\delta)-r(\varepsilon)\leq\frac{\delta}{\varepsilon}\, r(\varepsilon) \leq\frac{\delta}{\varepsilon}\, r(\eta)\,.
\end{equation}
The first and the last inequality are immediate as
$\mathcal{M}_{\varepsilon}\subsetneq\mathcal{M}_{\varepsilon
  +\delta}\subseteq \mathcal{M}_{\eta}$. For the second consider
$A= M + (\varepsilon+\delta)B\in \mathcal{M}_{\varepsilon +\delta}$
with $M\in \mathcal{M}$, $B\in \calB_V$.  As $\mathcal{M}$ is convex
and balanced we have
$\frac{\varepsilon}{\varepsilon +\delta}M = (\varepsilon+\delta)^{-1}
(\varepsilon M + \delta\hspace*{1pt} 0_{d\times d}) \in \mathcal{M}$
and so
$\frac{\varepsilon}{\varepsilon +\delta}A =
\frac{\varepsilon}{\varepsilon +\delta}M + \varepsilon B \in
\mathcal{M}_{\varepsilon}$. In particular, for any sequence
$A_1,\ldots, A_k\in \mathcal{M}_{\varepsilon +\delta}$, we have
\begin{equation*}
  \norm{A_k\cdots A_1}_2^{\frac{1}{k}}=\frac{\varepsilon +\delta}{\varepsilon}\norm{\left(\frac{\varepsilon}{\varepsilon +\delta}\right)A_k\cdots\left(\frac{\varepsilon}{\varepsilon +\delta}\right)A_1}_2^{\frac{1}{k}}
\end{equation*}
and taking the supremum and the limit as $k\rightarrow\infty$ we see
\begin{equation*}
  r(\varepsilon +\delta)=\rho\left(\mathcal{M}_{\varepsilon +\delta}\right)\leq\frac{\varepsilon +\delta}{\varepsilon}\rho\left(\mathcal{M}_{\varepsilon}\right)=\frac{\varepsilon +\delta}{\varepsilon}r(\varepsilon)\,.
\end{equation*}
The second inequality in \eqref{eq:eps_delta1} is now immediate.

Furthermore, for the given $\eta >0$, we have by Proposition~\ref{prop:hoel1} that there exists a constant $C_{\eta}>0$ such that
\begin{equation}
\label{eq:eps-delta2}
0\leq r(\varepsilon +\delta)-r(\varepsilon)\leq r(\varepsilon +\delta)-r(0)\leq C_{\eta}(\varepsilon+\delta)^{1/m}\,.
\end{equation}
In the following, we distinguish two cases:\\
(i) If $\delta^\frac{m}{m+1}\leq\varepsilon$, then we obtain from \eqref{eq:eps_delta1} that
	\begin{equation*}
	0\leq r(\varepsilon +\delta) - r(\varepsilon)\leq \frac{\delta}{\varepsilon} r(\eta)\leq \delta \delta^{-m/(m+1)}r(\eta)=\delta^{1/(m+1)}r(\eta)\,.
	\end{equation*}
(ii) If $0<\varepsilon\leq\delta^{m/(m+1)}$, then we may use \eqref{eq:eps-delta2} to get
	\begin{align*}
	0\leq r(\varepsilon + \delta)-r(\varepsilon)\leq C_{\eta}(\varepsilon +\delta)^{\frac{1}{m}}&\leq C_{\eta}\left(\delta^{m/(m+1)}+\delta\right)^{1/m}\\&\leq C_{\eta}\sqrt[m]{1+\eta^{1/(m+1)}}\delta^{1/(m+1)}\,.
	\end{align*}
This shows that $r$ is $\alpha$-Hölder continuous on $[0,\eta]$ with $\alpha=\frac{1}{m+1}$ and constant $C=\max\left\{r(\eta),C_{\eta}\sqrt[m]{1+\eta^{1/(m+1)}}\right\}$.
\end{proof}

\section{Pointwise Hölder continuity}
\label{sec:pointwise}

The aim of this section is to present our results on pointwise Hölder continuity of the joint spectral radius. This concerns the proof of the statement that the truth of Conjecture P2 would imply Conjecture P1. Less hypothetically, we will show pointwise Hölder continuity of the joint spectral radius with the less than expected constant $\alpha=\frac{1}{d+d^2}$. Two situations are closer to the ideal case: if $\rho(\calM)=0$, then the constant $\frac{1}{d}$ is achieved, and if  $\calM$ is finite, we obtain the nearly optimal $\frac{1}{d+\varepsilon}$ for all $\varepsilon>0$.

For our estimates we need the following elementary lemmata.
\begin{lem}
  \label{lem:exp-estimate}
    For all $x \in [0,1]$ and $k \geq 1$ we have
    $(1+ xk^{-1})^k \leq 1+2x$.
  \end{lem}
  \begin{proof}
    For all $x \geq 0$ we have
      $1+x/k \leq e^{x/k} $.
    It follows that for $x \in [0,1]$
  \begin{align*}
      \left(1+\frac{x}{k}\right)^k &\leq e^{x} \leq 1+2x.  \qedhere
  \end{align*}
  \end{proof}

\begin{lem}
  \label{lem:two-sequences}
  Let $\{a_n\}_{n\in\N}, \{b_n\}_{n\in\N}$ be two sequences of positive
  numbers such that $a_n> b_n$, $n\in\N$ and
  \[\liminf_{n \to \infty} nb_n^{n-1} > 1.\]
  Then for all sufficiently large $n$ it holds that $a_n^k - b_n^k \geq a_n-b_n$ for all $k \in \{1,\dots,n\}$.
\end{lem}

\begin{proof}
  We start by showing that
  for sufficiently large $n$ we have
  $kb_n^{k-1} \geq 1$ for all $k \in \{1,\dots,n\}$.
  Let $c >0$ and consider the function
  \[g(k) = \log (kc^{k-1}) = \log k + k\log c-\log c, \quad k \in (0,\infty).\]
  If $c\geq 1$, then $g$ is clearly increasing. For $0<c<1$, $g$ is unimodal with maximum at $- (\log c)^{-1}$.
  In both cases it follows that
  $g(k)\geq \min \{g(1), g(n)\}$ for all $k \in [1,n]$.
  This shows for all $c>0$ that $kc^{k-1} \geq \min\{1,nc^{n-1}\}$ for all
  $k \in [1,n]$.
  Using the condition that $nb_n^n >1$
  for sufficiently large $n$, there exists an $n_0 \in\N$ such that for all $n\geq n_0$ we have $kb_n^k \geq 1$ for all $k \in
  \{1,\dots,n\}$.

  Finally, for $n\geq n_0$ and $k \in \{1,\dots,n\}$
  \begin{align*}
    a_n^k - b_n^k &= (a_n - b_n)(a_n^{k-1} + a_{n}^{k-2}b_{n} +
                    \cdots + a_1b_{n}^{k-2}+b_n^{k-1}) \\
                  &\geq (a_n-b_n) k b_n^{k-1} \\
                  &\geq a_n-b_n. \qedhere
  \end{align*}
 \end{proof}

 The following lemma gives bounds for the deviation of perturbed products of finite length.

\begin{lem}
  \label{lem:norm-bound-for-perturbed-product}
  Let $\calM \in \calH(d)$ be a compact set of matrices
  with $\rho(\calM)=1$.
  Let $\Theta$ be the constant from Lemma~\ref{lem:polynomial-growth-bound} for
  the matrix set $\calM$.
  Then
  \begin{align*}
    \norm{(A_k+\varepsilon B_k)\cdots(A_1+\varepsilon B_1) - A_k
    \cdots A_1} \leq 2 \Theta^2\varepsilon k^{2d-1}
  \end{align*}
  for all $k\geq 1$, $\varepsilon<\frac{1}{\Theta}k^{-d}$, $A_i \in \calM$, $\norm{B_i}\leq 1$, $i=1,\ldots,k$.
\end{lem}
\begin{proof}
  Let $k\geq 1$, $A_j\in\calM$ and $B_j \in \K^{d \times d}$, $\norm{B_j}\leq 1$, $j=1,\ldots, k$ be arbitrary. When expanding the product in the norm, there remain products of length $k$ in which $m$ positions, $m=1,\ldots,k$, are taken by factors $\varepsilon B_i$. Let $1\leq i_1 < i_2 < \cdots < i_m \leq k$ be the indices of these positions. The corresponding product is
  \begin{equation}
      A_k\cdots A_{i_m+1}\varepsilon B_{i_m} A_{i_m-1}\cdots
  A_{i_2+1}\varepsilon B_{i_2}A_{i_2-1} \cdots A_{i_1+1}\varepsilon B_{i_1}A_{i_1-1}\cdots A_{1}.
  \end{equation}
  The norm of this product can be upper bounded by
\begin{align*}
  (\Theta j_0^{d-1}) \varepsilon (\Theta j_1^{d-1}) \dots \varepsilon (\Theta j_{m}^{d-1})
  &= \varepsilon^m \Theta^{m+1} (j_0 j_1 \cdots
    j_m)^{d-1} \\
  &\leq \varepsilon^m \Theta^{m+1} k^{(d-1)(m+1)}.
  \end{align*}
  where $j_s=i_{s+1}-i_{s}-1, \; s\in\{0,\dots,m\}$ denotes the length of the blocks
  of $A$ matrices in the product terminated by $B$ matrices
  and where $i_0=0, i_{m+1}=k+1$.

  To estimate the norm we estimate the sum of all such products and obtain
  \begin{align*}
      \norm{
    (A_k+\varepsilon B_k)\cdots(A_1+\varepsilon B_1)
    -
    A_k \cdots A_1} 
    &\leq \sum_{m=1}^k \binom{k}{m} \varepsilon^m (\Theta k^{d-1})^{m+1} \\
    &=\Theta k^{(d-1)}((1+\varepsilon \Theta k^{d-1})^k-1).
  \end{align*}

  By assumption $\varepsilon \Theta k^d <1$ and hence we can apply
  Lemma~\ref{lem:exp-estimate}
  with $x:=\varepsilon \Theta k^d$ to obtain 
    \begin{align*}
    \Theta k^{d-1}((1+\varepsilon \Theta k^{d-1})^k-1) &\leq \Theta
      k^{d-1} 2 \Theta k^d \varepsilon. \qedhere
    \end{align*}
  \end{proof}

  The next result links the approximability of $\rho$ by spectral radii of finite products with perturbation bounds of $\rho$. 

    \begin{thm}
  \label{thm:hoelder-lower-bound-allr}  
  Let $\calM \in \calH(d)$ be a compact set of matrices for which there are
      $r\geq 1$ and a constant $\Lambda>0$ such that for all $n\geq 1$
      it holds that
      \begin{align*}
    \rho(\calM)
    \leq \max_{1 \leq k \leq n} \max_{S \in \mathcal{S}_k(\calM)} \rho(S)^{1/k} +\frac{\Lambda}{n^r}\rho(\calM).
  \end{align*}
  Then there are an integer $n_0 \in \N$ and a constant $\Omega>0$, both only depending on 
  $\Lambda$ and not on $\calM$, and a constant $\tau>0$, depending on $\calM$,  such
  that for all $n \geq  n_0$ and all compact sets of matrices $\calN$
  with \begin{align*}
    d_H(\calM,\calN) \leq \frac{\tau}{n^{d^2+dr}}\rho(\calM)
  \end{align*}
  we have $\rho(\calN) \geq \rho(\calM)(1-\frac{\Omega}{n^r})$.
\end{thm}
\begin{proof}
  If $\rho(\calM)=0$, the conclusion is trivial.
  Otherwise we may scale $\calM$ and $\calN$ by $\rho(\calM)^{-1}$
  and assume $\rho(\calM)=1$ in the following.
  
  We will fix $n_0$ at the end of the proof. For the moment $n\geq 1$ is arbitrary.
  By assumption there are $k \in \{1,\dots,n\}$
  and matrices $A_1,\dots, A_k \in \calM$
  such that
  \begin{equation}
  \label{eq:finitespecradupperbound}
    1 \leq \rho(A_k \cdots A_1)^{1/k}+\frac{\Lambda}{n^r}.
  \end{equation}

  Let $\Theta$ be the constant
  from Lemma~\ref{lem:polynomial-growth-bound}
  applied to $\calM$.
  We may assume $\Theta\geq 1$.
  Set
  \begin{subequations}
  \begin{align}
    \label{eq:39-Delta}
    \Delta&:=2\Theta^2, \\
    \label{eq:39-Psi}
    \Psi&:=\max\left\{(2\Theta+\Delta)^{(d-1)/d}\Delta^{1/d},\Theta^{1/d}\Lambda\right\},\\
          \tau&:= \left(\frac{\Lambda}{\Psi}\right)^d,
     \label{eq:39-tau}     \\
    \varepsilon_n &:= \frac{\tau}{n^{d^2+dr}}.
  \end{align}
  \end{subequations}
  Let $\calN$ be a compact set of matrices with $d_H(\calM,\calN) \leq \varepsilon_n$.
  By definition of the Hausdorff metric there are matrices $B_1,\dots,B_k$
  of norm at most $1$
  such that 
  \[A_1+\varepsilon_n B_1,\dots,A_k+\varepsilon_n B_k \in \calN.\]
  We set $S:=A_k \cdots A_1$ and $T:=(A_k+\varepsilon_n B_k) \cdots (A_1+\varepsilon_n
    B_1)$.
  By Elsner's bound for the perturbed spectral radius of a matrix,
  Theorem~\ref{thm:elsner},
  we have
  \begin{align}
  \label{eq:rhoTk-ineqb}
    \rho(\calN)^k &\geq \rho(T)
    \geq \rho(S) -
      (2\norm{S}+\norm{T-S})^{(d-1)/d}\norm{T-S}^{1/d}.
  \end{align}
  Since $\varepsilon_n \leq \frac{1}{\Theta} n^{-(d^2+dr)} \leq \frac{1}{\Theta}
  k^{-d}$
  we can apply Lemma~\ref{lem:norm-bound-for-perturbed-product}
  to get an upper bound on $\norm{T-S}$, namely
  \begin{equation}
  \label{eq:38-help1b}
     \norm{T-S} \leq \Delta \varepsilon_n k^{2d-1}. 
  \end{equation}
  Furthermore $\norm{S} \leq \Theta k^{d-1}$ by Lemma~\ref{lem:polynomial-growth-bound}
  and $\varepsilon_n \leq k^{-d}$, hence
  \begin{equation}
  \label{eq:38-help2b}
    2\norm{S}+\norm{T-S} \leq (2\Theta + \Delta)k^{d-1}.
  \end{equation}
  Inserting \eqref{eq:finitespecradupperbound}, \eqref{eq:38-help1b}, \eqref{eq:38-help2b} into \eqref{eq:rhoTk-ineqb} we obtain
  \begin{align*}
  \rho(\calN)^k &\geq \rho(S) -  \Psi\cdot(k^{d-1})^{(d-1)/d} (\varepsilon_n k^{2d-1})^{1/d} \\
                  &\geq \left(1-\frac{\Lambda}{n^r}\right)^k - \Psi\cdot(k^{d^2}
                    \varepsilon_n)^{1/d} \\
                  &\geq \left(1-\frac{\Lambda}{n^r}\right)^k - \Lambda\cdot(n^{-dr})^{1/d}                  .  
  \end{align*}
  We obtain
    \begin{align*}
    \rho(\calN)&\geq \left(\left(1-\frac{\Lambda}{n^r}\right)^k -
                 \frac{\Lambda}{n^r}\right)^{1/k}.
    \end{align*}
    Finally, we apply Lemma~\ref{lem:two-sequences}
    with $a_n = 1 - \frac{\Lambda}{n^r}$
    and $b_n = 1-2\frac{\Lambda}{n^r}$, $n\geq 1$,
    to obtain the existence of an $n_0$ such that for all $n\geq n_0$, $1\leq k \leq n$
    \begin{equation}
    \label{eq:choose_n_0}
      \left(1-\frac{\Lambda}{n^r}\right)^k -
      \left(1-2\frac{\Lambda}{n^r}\right)^k \geq \frac{\Lambda}{n^r}.
    \end{equation}
    In particular $n_0$ only depends on $\Lambda$. For $n\geq n_0$  we get
    \begin{align*}
      \rho(\calN)&\geq 1 - 2\frac{\Lambda}{n^r}. 
    \end{align*}
    The proof is completed by setting $\Omega=2\Lambda$.
\end{proof}

With this result at hand, we obtain two pointwise Hölder continuity results. One will use the universal constant $\Lambda_d$ from \Cref{prop:converging-lower-bound}. For the second we recall the following result of Morris for finite matrix sets.

\begin{thm}[Morris {\cite[Theorem 1.2]{morris2010rapidly}}]
  \label{thm:morris-rapid-lower-bound}
  Let $\calM \subseteq \C^{d \times d}$ be a finite set of matrices
  and let $r \in \N$.
  There is a constant $\Lambda$ depending on $\calM$
  such that for all $n\geq 1$
  \begin{align*}
    \max_{1\leq k \leq n} \sup_{S \in \calS_k(\calM)} \rho(S)^{1/k} \geq \rho(\calM)\left(1- \frac{\Lambda}{n^r} \right). 
  \end{align*}
\end{thm}

We can now state our main results of this section. We begin by observing that if conjecture P2 is satisfied for a certain $\alpha >0 $, then P1 is satisfied for 
$\alpha' = \min\{\alpha, \frac{1}{d}\}$.

\begin{thm}
    \label{thm:pointwiseHoelder-v1}
    Let $\alpha >0$.
    Assume that for every $\calM \in \calH(d)$
    there are $\Gamma>0$ and $1>\gamma>0$ such that
    \[\rho(\calN) \geq \rho(\calM) - \Gamma d(\calM,\calN)^\alpha\]
    for all $\calN \in \calH(d)$ with $d(\calM,\calN)<\gamma$.
    Then the joint spectral radius in dimension $d$ is pointwise Hölder continuous with exponent $\alpha' = \min\{\alpha, \frac{1}{d}\}$.
\end{thm}
\begin{proof}
  Fix $\mathcal{M}\in \calH(d)$.
  Let $\calB_2$ be the unit ball
  in $\C^{d\times d}$ with respect to the spectral norm, set $\eta:=1$ and let
  $C_\calM$ be the constant guaranteed to exist by \Cref{prop:hoel1}
  for the map $r:\varepsilon \mapsto \rho(\calM+\varepsilon \calB_2)$.
  Now
  consider an arbitrary $\calN\in\calH(d)$ with
  $d_H(\calM,\calN)\leq \eta$. 
  On one hand \Cref{prop:hoel1} yields
 \begin{align*}
      \rho(\calN) - \rho(\calM)  &\leq
      \rho(\calM + d_H(\calM,\calN)\calB_2) - 
                                     \rho(\calM) \\
                                    &\leq C_{\calM} d_H(\calM,\calN)^{1/d} \leq C_{\calM} d_H(\calM,\calN)^{\alpha'}.
  \end{align*} 
  On the other hand, by assumption if $d(\calM,\calN)<\gamma$ then
  \begin{align*}
    \rho(\calM) - \rho(\calN) \leq
    \Gamma d_H(\calN,\calM)^{\alpha}\leq
    \Gamma d_H(\calN,\calM)^{\alpha'}. 
  \end{align*}
  Together these two inequalities show
  \begin{align*}
        \limsup_{\calN \to \calM} \frac{\abs{\rho(\calM)-\rho(\calN)}}{d_H(\calM,\calN)^{\alpha'}} &\leq \max(\Gamma,C_\calM) < \infty.\qedhere
  \end{align*}
\end{proof}

We end this section with the two concrete positive results on pointwise Hölder continuity: On $\calH(d)$ we have pointwise Hölder continuity with exponent $\alpha=\frac{1}{d^2+d}$. If we restrict to the set of finite subsets of $\calH(d)$ then pointwise Hölder continuity with exponent $\alpha=\frac{1}{d+\varepsilon}$ holds for every $\varepsilon>0$. Note that this does not imply pointwise Hölder continuity with exponent $\frac{1}{d}$.

\begin{thm}
    \label{thm:pointwise-Hoelder-everywhere-v1}
    \begin{enumerate}[(i)]
        \item If $\calM\in \calH(d)$ satisfies $\rho(\calM)=0$, then the joint spectral radius is Hölder continuous at $\calM$ with exponent 
        $\alpha=\frac{1}{d}$.
        \item The joint spectral radius in dimension $d$ is pointwise Hölder continuous on $\calH(d)$ with exponent $\alpha=\frac{1}{d^2+d}$.
        \item Let $\calM\subseteq \C^{d\times d}$ be finite. For every $\varepsilon>0$ the joint spectral radius is Hölder continuous at $\calM$ with exponent 
        $\alpha=\frac{1}{d+\varepsilon}$.
    \end{enumerate}
\end{thm}

\begin{proof}
    (i) This is a consequence of \Cref{prop:hoel1}.
    
    (ii) Let $\calM\in\calH(d)$ be arbitrary. By (i) we may assume $\rho(\calM)>0$. Let $r=1$ and $\Lambda_d$ be the constant depending only on $d$ from \Cref{prop:converging-lower-bound}. With these constants the assumptions of \Cref{thm:hoelder-lower-bound-allr} are satisfied. 
    Using this result and the constants $\Omega$ and $\tau$
      obtained from it, we may apply \Cref{lem:Hoelder-via-integers} (with
    $p=r=1$, $q=d^2+d$) to obtain
    \begin{equation*}
        \rho(\calM) \geq \rho(\calN) - \left(2\Omega
        \rho(\calM)^{1-1/(d^2+d)}\tau^{-1/(d^2+d)}\right) d(\calM,\calN)^{1/(d^2+d)}
    \end{equation*}
    for all $\calN \in \calH(d)$ with $d(\calM,\calN) \leq \tau n_0^{-(d^2+d)}\rho(\calM)$.
    The claim follows from \Cref{thm:pointwiseHoelder-v1}.

    (iii) Fix $\varepsilon>0$ and choose an integer $r> \frac{d^2}{\varepsilon}$. As $\calM$ is finite, we may apply \Cref{thm:morris-rapid-lower-bound} to find a constant $\Lambda>0$ such that the assumptions of \Cref{thm:hoelder-lower-bound-allr} are satisfied for $\Lambda,r$. As in (ii) we may apply \Cref{lem:Hoelder-via-integers} (this time with $p=r$, $q=d^2+dr$) to get that $\rho$ is Hölder continuous at $\calM$ with exponent $\frac{r}{d^2+dr} \geq \frac{1}{d + \varepsilon}$.
\end{proof}

\section{Local $\frac{1}{d^2+d}$-Hölder-continuity of the JSR in dimension two}
\label{sec:two-dim}

In this section we will first clarify the relations between our local conjectures L1, L2, L3. \Cref{lem:independent-lower-bound-implies-locally-hoelder} below shows that that the veracity of Conjecture L2 implies that L1 holds, where the Hölder exponent of $\rho$ is directly determined by the exponent for which L2 is true. Then \Cref{lem:independent-product-bound-implies-independent-lower-bound} shows that L3 implies L2 with constant $\alpha=\frac{1}{d+d^2}$. We stress that we do not know whether L3 holds in general.

However, in dimension $2$ we can improve \Cref{lem:polynomial-growth-bound}
by removing the dependence of the constant $\Theta$ on $\calM$, or in other words L3 is true in dimension $2$.
By the previous general results this shows local Hölder continuity of the joint spectral radius on $\calH_{>0}(2)$. This is our final result of the section in \Cref{cor:local-hoelder-dim-2}.

We start with the general considerations. The following lemma shows that L2 implies L1. The statement is for general $\alpha$, but recall that of course Hölder continuity with coefficient $\alpha>1$ implies that the function is constant.

\begin{lem}
  \label{lem:independent-lower-bound-implies-locally-hoelder}
  Let $\mathcal{K} \subseteq \calH(d)$ be compact.
  If there are $\alpha>0$ and $\Gamma>0$ such that
  for all $\mathcal{M}, \mathcal{N} \in \mathcal{K}$ we have 
  \begin{align*}
    \rho(\calM) \geq \rho(\mathcal{N}) - \Gamma d(\mathcal{M},\mathcal{N})^{\alpha}
  \end{align*}
  then $\rho$ is $\alpha$-Hölder continuous on $\mathcal{K}$.
\end{lem}
\begin{proof}
  Under the given assumptions we have
  \begin{align*}
    \Gamma d(\calM,\calN)^\alpha &\geq \rho(\calN) - \rho(\calM), \\
    \Gamma  d(\calM,\calN)^\alpha&\geq \rho(\calM) - \rho(\calN)
  \end{align*}
  for all $\calM, \calN \in \mathcal{K}$ 
  and therefore
  \begin{align*}
    \sup_{\calM,\calN \in \mathcal{K}, \calM \neq \calN}
    \frac{\abs{\rho(\calM)-\rho(\calN)}}{d(\calM,\calN)^\alpha} &\leq
    \Gamma. \qedhere
  \end{align*}
\end{proof}

The following theorem shows that 
uniform trajectory bounds can be used to obtain Hölder continuity 
of $\rho$. When Conjecture L3 is satisfied, the
assumptions of the theorem 
hold for all $\mathcal{K} \subseteq \mathcal{H}_{>0}(d)$.
Notice that they cannot hold for every $\mathcal{K} \subseteq \mathcal{H}(d)$
as \Cref{example:L3isintrouble} shows.
\begin{thm}
  \label{lem:independent-product-bound-implies-independent-lower-bound}
  Let $\mathcal{K} \subseteq \mathcal{H}(d)$ be compact.
  Assume there is a constant $\Theta$ such that
  for all $\mathcal{M} \in \mathcal{K}$ with $\rho(\calM)>0$
  we have
  \begin{align*}
    \sup_{S \in \calS_k(\calM)} \norm{S} \leq \Theta  k^{d-1} \rho(\mathcal{M})^k.
  \end{align*}
  Then there is a constant $\Gamma>0$ such that 
  \begin{align*}
    \rho(\mathcal{N}) \geq \rho(\mathcal{M}) - \Gamma d(\mathcal{M},\mathcal{N})^{1/(d^2+d)}
  \end{align*}
  for all compact subsets $\calM, \calN \in \mathcal{K}$.

  In particular, $\rho$ is Hölder continuous on $\calK$ with exponent $\alpha=\frac{1}{d^2+d}$.
\end{thm}
\begin{proof}
  By inspecting the proof of \Cref{thm:hoelder-lower-bound-allr}
  we see that under our assumptions
  the constants $\tau$ and
  $\Omega$ in \Cref{thm:hoelder-lower-bound-allr}
  can be chosen independently of $\calM$ and $\calN$.
  More precisely let $\mathcal{K}$ and $\Theta$ be as in the assumptions.
  Let $\Lambda_d$ be the constant from
  \Cref{prop:converging-lower-bound}
  and set again 
  \begin{align*}
    \Delta &:=2\Theta^2,\\
    \Psi &:=\max\left\{(2\Theta+\Delta)^{(d-1)/d}\Delta^{1/d},\Theta^{1/d}\Lambda_d\right\},\\
    \tau &:= \left(\frac{\Lambda_d}{\Psi}\right)^d,\\
    \varepsilon_n &:= \frac{\tau}{n^{d^2+d}}, \\
    \Omega &:=2\Lambda_d.
  \end{align*}
  Note that under the assumptions all constants only depend on $\calK$.
  Then the proof of \Cref{thm:hoelder-lower-bound-allr}
  shows that there is an integer $n_0$ only depending on
  the dimension $d$, such that for all $n \geq  n_0$ and all
  $\calM, \calN \in \mathcal{K}$
  with \begin{align*}
    d_H(\calM,\calN) \leq \frac{\tau}{n^{d^2+d}}\rho(\calM)
  \end{align*}
  we have 
  \[\rho(\calN) \geq \rho(\calM)\left(1-\frac{\Omega}{n}\right) .\]
  Set $R := \max\{ \rho(\calM) \setsep \calM\in \calK\}$.
  As before, we get with \Cref{lem:Hoelder-via-integers} that
  \begin{equation}
 \label{eq:lH-cond1}  
 \begin{aligned}     
 \rho(\calN) &\geq \rho(\calM) - \left(2\Omega\rho(\calM)^{1-1/(d^2+d)} \tau^{-1/(d^2+d)}\right) d_H(\calM,\calN)^{1/(d^2+d)} \\
     &\geq \rho(\calM) - \left(2\Omega R^{1-1/(d^2+d)} \tau^{-1/(d^2+d)}\right) d_H(\calM,\calN)^{1/(d^2+d)},
     \end{aligned}
  \end{equation}
  for all $\calM, \calN \in \mathcal{K}$ with $d_H(\calM,\calN) \leq
  \frac{\tau}{n_0^{d^2+d}} \rho(\calM)$. On the other hand for $\calM, \calN \in \mathcal{K}$ with
  $d_H(\calM,\calN) \geq \frac{\tau}{n_0^{d^2+d}}\rho(\calM)$
  we have
  \begin{equation}
  \label{eq:lH-cond2}
      \rho(\calN) \geq 0 \geq \rho(\calM) - \frac{n_0^{d^2+d}}{\tau} d_H(\calM,\calN).
  \end{equation}
  Note that this also covers the case $\rho(\calM)=0$.
  Combining \eqref{eq:lH-cond1}, \eqref{eq:lH-cond2}, this shows for all $\calM,\calN \in \mathcal{K}$ that
  \begin{align*}
    \rho(\calN) \geq \rho(\calM) - \Gamma d_H(\calM,\calN)
  \end{align*}
  with $\Gamma:=\max\left\{\frac{n_0^{d^2+d}}{\tau},2
  \Omega R^{1-1/(d^2+d)} 
  \tau^{-1/(d^2+d)} \right\}$.
  Local Hölder continuity is now an immediate consequence of \Cref{lem:independent-lower-bound-implies-locally-hoelder}.
\end{proof}

We now turn to the two dimensional case, in which the constants for the exponential-polynomial bounds can be locally bounded. The proof for this relies to a great part on the geometry of $\C^2$. We first need the following observation. 

\begin{lem}
    \label{lem:distance-subspace}
    Consider $r_1\geq r_2 > 0$
    and let $(\eta, \zeta) \in \C^2$ be a point 
    with $\abs{\eta} \leq r_1$ and $\abs{\zeta}=r_2$.
    Then the distance between the origin in $\C^2$
    and the affine subspace spanned by $(r_1,0)$
    and $(\eta,\zeta)$ is at least $\frac{1}{\sqrt{5}}r_2$.
\end{lem}
\begin{proof}
    Let $V \subseteq \C^2$ be the affine subspace spanned by $(r_1,0)$
    and $(\eta,\zeta)$.
    The orthogonal projection
    of $(0,0)$ onto $V$
    is 
    \begin{align*}
        \begin{pmatrix} r_1 \\ 0 \end{pmatrix}+
        \frac{r_1(\overline{\eta}-r_1)}{\abs{r_1-\eta}^2+\abs{\zeta}^2} \begin{pmatrix}r_1-\eta \\ -\zeta\end{pmatrix}
        = \frac{r_1}{\abs{r_1-\eta}^2+\abs{\zeta}^2} \begin{pmatrix}\abs{\zeta}^2 \\ \zeta(r_1-\overline{\eta})\end{pmatrix}.
    \end{align*}
    Thus the distance between $V$ and $(0,0)$
    satisfies
    \begin{align*}
        \dist(V,(0,0))^2 &= \frac{r_1^2\abs{\zeta}^2}{(\abs{r_1-\eta}^2+\abs{\zeta}^2)^2}(\abs{r_1-\eta}^2+\abs{\zeta}^2) \\
        &= \frac{r_1^2 r_2^2}{\abs{r_1-\eta}^2+r_2^2}.
    \end{align*}
   Since $\abs{\eta} \leq r_1$ we have
   \begin{align*}
    \dist(V,(0,0)) &\geq \frac{r_1 r_2}{\sqrt{4r_1^2+r_2^2}} 
    =\frac{r_2}{\sqrt{4+\frac{r_2^2}{r_1^2}}}.
   \end{align*}
   Finally $0< r_2 \leq r_1$ implies
   \begin{align*}
        \dist(V,(0,0)) &\geq \frac{1}{\sqrt{5}}r_2. \qedhere
   \end{align*}
\end{proof}

With this preparation we can prove Conjecture L3 in dimension two. We
will use the Frobenius norm (also called Hilbert-Schmidt norm), see
\cite[Section II.3.1]{stewartMatrixPerturbationTheory1990}.

\begin{lem}
  \label{lem:polynomial-growth-bound-plane}
  Let $L\geq 1$ and let $\calM \subseteq \C^{2 \times 2}$ be a compact set of matrices
  whose Frobenius norm is at most $L$. Assume $\rho(\calM)=1$.
  Then for all $k \in \N$ we have
  \begin{align*}
    \norm{A_k \cdots A_1}_2 \leq 6Lk
  \end{align*}
  for $A_i \in \calM$, $i=1,\ldots,k$.
\end{lem}

\begin{proof}
    First we consider the case of reducible $\calM$, which for a change
    is easier here.
    In this case there is a unitary matrix $U$
    such that $U^* \calM U$ contains only upper triangular $2 \times 2$ matrices.
    Since the Frobenius norm is unitarily invariant,
    the entries of the matrices in $U^* \calM U$
    have absolute value at most $L$. Furthermore
    $\rho(U^* \calM U)=\rho(\calM)=1$, so all
    diagonal entries of $A \in U^*\calM U$ have absolute value
    bounded from above by $1$ and an upper right off-diagonal entry
    whose absolute value is bounded by $L$.
    Let $x \in \C^2$ and $B \in U^* \calM U$.
    By the previous considerations $\abs{(Bx)_2} \leq \abs{x_2}$ and 
    $\abs{(Bx)_1} \leq \abs{x_1}+ L \abs{x_2}$.
    By induction this shows for $B_i \in U^* \calM U$,
    $i=1,\ldots,k$ that
    \[\abs{(B_k \cdots B_1 x)_2} \leq \abs{x}_2\]
    and 
    \[
    \abs{(B_k \cdots B_1 x)_1} \leq \abs{x}_1+k L \abs{x}_2.
    \]
    This shows 
    \begin{align*}
    \norm{B_k \cdots B_1 x}_2 \leq (1+kL) \norm{x}_2 \leq 2kL \norm{x}_2.
    \end{align*}
    Since $\norm{\cdot}_2$ is unitarily invariant, we have for arbitrary $A_i \in \calM$,
    $i=1,\ldots,k$,
    \begin{align*}
    \norm{A_k \cdots A_1}_2 \leq 2kL.
    \end{align*}

    Now consider the irreducible case.
    Let $v$  be an extremal norm of $\calM$ with unit ball $\calB_v := \{x \in \C^2 \setsep v(x)\leq 1\}$.
    
    Let $\tilde{y} \in \calB_v$ be a vector 
    maximizing $\norm{\tilde{y}}_2$ and
    set $y = \tilde{y}/\norm{\tilde{y}}_2$.
    Let $z\in \C^2$ with $\norm{z}_2=1$ be
    perpendicular to $y$.
    Let $U= \begin{pmatrix}
        y & z
    \end{pmatrix}$ be the unitary matrix with columns
     $y$ and $z$ and define a norm by $w(x) := v(Ux)$, $x\in\C^2$.
    This is an extremal norm for
    $U^* \calM U$ with unit ball $\calB_w :=U^* \calB_v$.
    Denote by $e_1,e_2$ the standard basis of $\C^2$.
    Set 
    \begin{align*}
        r_1&:=\max \{\abs{x_1} \setsep x \in \calB_w\},\\
        r_2&:=\max \{\abs{x_2} \setsep x \in \calB_w\}.
    \end{align*}
    By definition $re_1=U^* \tilde{y}\in \calB_w$ for some $r>0$. As $U^* \tilde{y}$ maximizes the Euclidean norm in $\calB_w$ it follows that $r=r_1$ and so
    $r_1 e_1\in\partial\calB_w$. Furthermore,
    \[\calB_w \subseteq \{ x \in \C^2 \setsep \abs{x_1} \in [-r_1,r_1], \abs{x_2} \in [-r_2,r_2]\}.\]
    
    Consider $\tilde{A} = \begin{pmatrix} a & b \\ c &d\end{pmatrix} \in 
    U^* \calM U$. Since $r_1e_1 \in \partial\calB_w$, we have
    $w(r_1e_1) = 1 \geq w(r_1 \tilde{A}e_1)$, as $w$ is extremal. It follows that
    $\abs{(r_1\tilde{A}e_1)_1} \leq r_1$
    and so $\abs{a} \leq 1$.
    Since the Frobenius norm is unitarily invariant, we have $\abs{b}\leq L$. Let $x \in \C^2$ with
    $w(x)\leq 1$. Then
    $w(\tilde{A}x) \leq 1$
    implies
    \begin{align*}
        \abs{(\tilde{A}x)_2} &\leq r_2\\
        \abs{(\tilde{A}x)_1} &\leq \abs{x_1} + Lr_2
    \end{align*}
    By induction we obtain for an arbitrary sequence
    $\tilde{A}_1,\ldots,\tilde{A}_k$ in $U^*\calM U$ that
    \begin{align*}
        \abs{(\tilde{A}_k \cdots\tilde{A}_1 x)_2} &\leq r_2 \\
        \abs{(\tilde{A}_k \cdots \tilde{A}_1 x)_1} &\leq  \abs{x_1} + kLr_2.
    \end{align*}
    
     To calculate the norm of a product
    $\tilde{A}_k \cdots \tilde{A}_1$, $\tilde{A}_i \in U^*\calM U$, $i=1,\ldots,k$, we only
    have to consider the norm
    of points $\tilde{A}_k \cdots \tilde{A}_1 \tilde{x}$ with $\tilde{x}$ on the boundary of $\calB_w$.
    There is a point $\tilde{z}=(\eta,\zeta)$ in the boundary of $\calB_w$
    with $\abs{\zeta}=r_2$.
    
    Since $r_1e_1, \tilde{z}$ is a basis of $\C^2$, there are $\alpha,\beta \in [0,1]$, $\alpha+\beta=1$, $\gamma \in \C$ with $\abs{\gamma}=1$
    and $\delta \in \C$ such that 
    $\alpha r_1e_1 + \beta \gamma \tilde{z}  = \delta \tilde{x}$.
    Now $\tilde{x}$ is contained in the boundary of $\calB_w$
    and $\alpha r_1e_1+\beta \gamma \tilde{z}$ is contained
    in $\calB_w$, hence \[\abs{\delta}=\abs{\delta}w(\tilde{x})=w(\delta\tilde{x})\leq 1.\]
    Furthermore $\delta \tilde{x}$ is contained in the 
    affine span of $r_1 e_1$ and $\gamma \tilde{z}$, so by
    \Cref{lem:distance-subspace} we have
    $\norm{\tilde{x}}_2\geq \norm{\delta\tilde{x}}_2\geq \frac{1}{\sqrt{5}} r_2$.
    Therefore
    \begin{align*}
        \frac{\norm{\tilde{A}_k\cdots\tilde{A}_1 \tilde{x}}^2_2}{\norm{\tilde{x}}^2_2}
        &\leq  \frac{r_2^2+(\norm{\tilde{x}}_2+kLr_2)^2}{\norm{\tilde{x}}_2^2} \\
        &=1 + \frac{r_2^2(1+k^2L^2)}{\norm{\tilde{x}}_2^2}+ \frac{2kLr_2}{\norm{\tilde{x}}_2}\\
        &\leq 1+5(1+k^2L^2)+2\sqrt{5}kL\\
        &\leq (1+\sqrt{5}(kL+1))^2\\
        &\leq (6kL)^2
    \end{align*}
    for $\tilde{A}_i \in U^* \calM U$, $i=1,\ldots,k$.
    Again by the unitary invariance of  the Euclidean norm we also have
    \begin{align*}
        \norm{A_k \cdots A_1}_2 \leq 6kL
    \end{align*}
    for arbitrary $A_i \in \calM$, $i=1,\ldots,k$.
  \end{proof}

With the previous estimate at hand, we obtain the following result on Hölder continuity on $\calH_{>0}(2)$.
  \begin{cor}
    \label{cor:local-hoelder-dim-2}
    The joint spectral radius is locally
    $\frac{1}{6}$-Hölder-continuous
    on $\mathcal{H}_{>0}(2)$.
  \end{cor}
  \begin{proof}
    Let $\mathcal{K} \subseteq \mathcal{H}_{>0}(2)$ be compact.
    There is a constant $L>0$ such that
    the Frobenius norm of all matrices in $\mathcal{K}$ is
    at most $L$. In particular, $\rho(\calM)\leq L$ for all $\calM\in \calK$.
    Also, by continuity of the joint spectral radius, $\min \{ \rho(\calM)\setsep \calM\in \calK\} >0$. 
    By
    \Cref{lem:polynomial-growth-bound-plane} we have for all $\calM\in\calK$ and $S\in\calS_k(\calM)$ that
    \begin{equation*}
        \norm{S}_2 \leq \frac{6L}{\rho(\calM)}\rho(\calM)^k k \leq \frac{6L}{\min \{ \rho(\calM)\setsep \calM\in \calK\}} \rho(\calM)^k k.
    \end{equation*}
    Combined with
    \Cref{lem:independent-product-bound-implies-independent-lower-bound}
    this shows the existence of a constant $\Gamma$ such that
    \begin{align*}
      \rho(\calN) \geq \rho(\calM) - \Gamma d_H(\calM,\calN)^{1/(2^2+2)}
    \end{align*}
    for all $\calM,\calN \in \mathcal{K}$.
    Together with
    \Cref{lem:independent-lower-bound-implies-locally-hoelder}
    this shows that the joint spectral radius is
    $\frac{1}{6}$-Hölder-continuous on $\mathcal{K}$.
  \end{proof}

\section{Continuous-Time}
\label{sec:continuous}

In this section we briefly discuss the consequences of our results for continuous-time linear inclusions and their exponential growth rates. Given a compact set $\calM\in \calH(d)$ the differential inclusion
\begin{equation}
    \label{eq:ldi}
    \dot x \in \{ Ax \setsep A \in \calM \}
\end{equation}
gives rise to a set of evolution operators of the time-varying differential equations
\begin{equation}
\label{eq:cttv}
    \dot x(t) = A(t) x(t),
\end{equation}
where $A(\cdot)$ is a measurable function on $[0,\infty)$ taking values in $\calM$. In fact, the solutions to the differential inclusion \eqref{eq:ldi}
are precisely the solutions to a linear equation of the type \eqref{eq:cttv} for a suitable $A(\cdot)$, see \cite{shorten2007stability} for details. Thus, to $\calM$ we can associate the set of time $t$ evolution operators
\begin{equation}
\begin{aligned}
    \calS_{t}(\calM,\R) := &\left\{ \Phi_{A(\cdot)}(t,0) \setsep \exists \ A(\cdot) \in L^\infty(\R_+,\calM) : \Phi_{A(\cdot)}(\cdot,\cdot) \text{ is the } \right. \\ &\left. \phantom{\ensuremath{\Phi_{A(\cdot)}(t,0) \setsep }} \text{ evolution operator for \eqref{eq:cttv}  }  \right\}.
\end{aligned}
\end{equation}

With this definition the joint spectral radius $\rho(\calM,\R)$ for the continuous time case may be defined as in \eqref{eq:jsr}, although it is more common to study the associated maximal Lyapunov exponent $\lambda(\calM) := \log (\rho(\calM,\R))$. It is easy to see that if we set $\hat{\calM}:= \calS_{1}(\calM,\R)$, then
\begin{equation}
    \rho(\calM,\R) = \rho(\hat{\calM},\N).
\end{equation}
The map $\calM \mapsto \cl \calS_{1}(\calM,\R)$ is locally Lipschitz
continuous from $\calH(d)$ to $\calH_{>0}(d)$. In particular, the
implicit time-discretisation of a continuous-time linear inclusion
always results in a discrete-time system with positive joint spectral
radius. This is essentially a consequence of the invertibility of
transition operators in the continuous-time case.  With these
observations we obtain the following assertions from the results of
the previous sections.

\begin{thm}
    \label{thm:localHoelder-continoustime}
The maximal Lyapunov exponent is pointwise Hölder continuous on $\calH(d)$ with exponent $\alpha=\frac{1}{d^2+2}$.   
\end{thm}

\begin{thm}
    \label{thm:2dimHoelder-continoustime}
The maximal Lyapunov exponent is locally $\frac{1}{6}$-Hölder continuous on $\calH(2)$.   
\end{thm}

\section{Conclusion}
\label{sec:conclusion}

We have investigated regularity properties of the joint spectral
radius as a map defined on the space of compact subsets of the set of
real or complex $d\times d$-matrices. In general, we obtain pointwise
Hölder continuity of order $\frac{1}{d^2+d}$. For finite matrix sets
this constant reduces to order $\frac{1}{d+\varepsilon}$. In general,
we conjecture that the joint spectral radius is even locally Hölder
continuous and in analogy to the ordinary spectral radius that the
order is $\frac{1}{d}$. This conjecture is still elusive. For the case
$d=2$ it is shown that local Hölder continuity of order $\frac{1}{6}$
holds. Unfortunately, the argument for $d=2$ does not lend itself to a
straightforward inductive argument by which a similar result holds in
higher dimensions. Indeed, we are able to prove that the particular
straightforward induction that we could think of cannot be
performed. For reasons of space, this will be discussed elsewhere. In
general, the investigation of the regularity of the joint spectral
radius continues to be an interesting topic.

\setcounter{biburllcpenalty}{100}
\setcounter{biburlucpenalty}{100}
\setcounter{biburlnumpenalty}{100}

\biburlnumskip=0mu plus 1mu\relax
\biburlucskip=0mu plus 1mu\relax
\biburllcskip=0mu plus 1mu\relax
\printbibliography

@book{stewartMatrixPerturbationTheory1990,
  title = {Matrix Perturbation Theory},
  author = {Stewart, Gilbert W. and Sun, Ji-guang},
  year = {1990},
  publisher = {Academic Press},
  location = {{Boston etc.}},
  isbn = {978-0-12-670230-9},
  langid = {english}
}

@book{katoPerturbationTheoryLinear1995,
  title = {Perturbation Theory for Linear Operators},
  author = {Katō, Tosio},
  year = {1995},
  series = {Classics in Mathematics},
  publisher = {{Springer}},
  location = {{Berlin}},
  isbn = {978-3-540-58661-6},
  langid = {english},
  pagetotal = {619},
  doi = {10.1007/978-3-642-66282-9},
}

@misc{varneyMarginalGrowthRates2022,
  title = {On marginal growth rates of matrix products},
  author = {Varney, Jonah and Morris, Ian D.},
  year = {2022},
  month = sep,
  number = {arXiv:2209.00449},
  eprint = {2209.00449},
  primaryclass = {math},
  publisher = {{arXiv}},
  doi = {10.48550/arXiv.2209.00449},
  urldate = {2023-04-26},
  archiveprefix = {arxiv}
}

@article{guglielmiAsymptoticPropertiesFamily2001,
  title = {On the asymptotic properties of a family of matrices},
  author = {Guglielmi, N. and Zennaro, M.},
  year = {2001},
  month = jan,
  journal = {Linear Algebra and its Applications},
  volume = {322},
  number = {1-3},
  pages = {169--192},
  doi = {10.1016/S0024-3795(00)00228-7},
  urldate = {2023-05-02},
  langid = {english}
}

@article{Wirt02,
  title = {The generalized spectral radius and extremal norms},
  author = {Wirth, Fabian},
  year = {2002},
  month = feb,
  journal = {Linear Algebra and its Applications},
  volume = {342},
  number = {1},
  pages = {17--40},
  doi = {10.1016/S0024-3795(01)00446-3},
}

@BOOK{Jungers,
 author = {Jungers, R.},
 title = {The Joint Spectral Radius: Theory and Applications},
 year = {2009},
 publisher = {Springer-Verlag},
 doi = {10.1007/978-3-540-95980-9}
}

@ARTICLE{Bara88,
       AUTHOR             = {N. E. Barabanov},
       JOURNAL            = {Autom.~Remote Control},
       PAGES              = {152--157, 283-287, 558--565},
       TITLE              = {Lyapunov indicator of discrete inclusions. {I--III}},
       VOLUME             = {49},
       YEAR               = {1988}
}

@article{Elsn95,
  title = {The generalized spectral-radius theorem: {{An}} analytic-geometric proof},
  shorttitle = {The Generalized Spectral-Radius Theorem},
  author = {Elsner, L.},
  year = {1995},
  month = apr,
  journal = {Linear Algebra and its Applications},
  series = {Nonnegative {{Matrices}}, {{Applications}} and {{Generalizations}} and the {{Eighth Haifa Matrix Theory Conference}}},
  volume = {220},
  pages = {151--159},
  doi = {10.1016/0024-3795(93)00320-Y},
}

@article{BergWang92,
  title = {Bounded semigroups of matrices},
  author = {Berger, Marc A. and Wang, Yang},
  year = {1992},
  month = mar,
  journal = {Linear Algebra and its Applications},
  volume = {166},
  pages = {21--27},
  doi = {10.1016/0024-3795(92)90267-E},
}

@article{Gurv95,
  title = {Stability of discrete linear inclusion},
  author = {Gurvits, Leonid},
  year = {1995},
  month = dec,
  journal = {Linear Algebra and its Applications},
  volume = {231},
  pages = {47--85},
  doi = {10.1016/0024-3795(95)90006-3},
}

@article{KOZYAKIN201012,
  title = {An explicit {{Lipschitz}} constant for the joint spectral radius},
  author = {Kozyakin, Victor},
  year = {2010},
  month = jul,
  journal = {Linear Algebra and its Applications},
  volume = {433},
  number = {1},
  pages = {12--18},
  doi = {10.1016/j.laa.2010.01.028},
}

@Article{RotaStra60,
 Author = {Rota, Gian-Carlo and Strang, W. Gilbert},
 Title = {A note on the joint spectral radius},
 FJournal = {Nederlandse Akademie van Wetenschappen. Proceedings. Series A. Indagationes Mathematicae},
 Journal = {Nederl. Akad. Wet., Proc., Ser. A},
 Volume = {63},
 Pages = {379--381},
 Year = {1960},
 }

@article{morris2010rapidly,
  title={A rapidly-converging lower bound for the joint spectral radius via multiplicative ergodic theory},
  author={Morris, Ian D.},
  journal={Advances in Mathematics},
  volume={225},
  number={6},
  pages={3425--3445},
  year={2010},
  doi = {10.1016/j.aim.2010.06.008},
}

@article{bochi2003inequalities,
  title = {Inequalities for numerical invariants of sets of matrices},
  author = {Bochi, Jairo},
  year = {2003},
  month = jul,
  journal = {Linear Algebra and its Applications},
  volume = {368},
  pages = {71--81},
  doi = {10.1016/S0024-3795(02)00658-4},
}

@incollection{breuillard2022joint,
  title={On the joint spectral radius},
  author={Breuillard, Emmanuel},
  booktitle={Analysis at Large: Dedicated to the Life and Work of Jean Bourgain},
  editor={Avila, Artur and Rassias, Michael Th and Sinai, Yakov},
  pages={1--16},
  year={2022},
  publisher={Springer},
  doi = {10.1007/978-3-031-05331-3_1},
}

@article{bezerraUpperBoundRegularity2023,
  title = {Upper bound on the regularity of the {L}yapunov exponent for random products of matrices},
  author = {Bezerra, Jamerson and Duarte, Pedro},
  year = {2023},
  volume = {403},
  pages =  {829--875},
  month = aug,
  journal = {Communications in Mathematical Physics},
  doi = {10.1007/s00220-023-04815-5},
  urldate = {2023-08-16},
  langid = {english}
}

@book{baumgartelAnalyticPerturbationTheory1985,
  title = {Analytic {{Perturbation Theory}} for {{Matrices}} and {{Operators}}},
  author = {Baumg{\"a}rtel, Hellmut},
  year = {1985},
  publisher = {Birkh\"auser},
  address = {Basel},
  isbn = {978-3-7643-1664-8},
  langid = {english},
}

@misc{avilaContinuityLyapunovExponents2023,
  title = {Continuity of the {{Lyapunov}} exponents of random matrix products},
  author = {Avila, Artur and Eskin, Alex and Viana, Marcelo},
  year = {2023},
  month = may,
  number = {arXiv:2305.06009},
  eprint = {2305.06009},
  primaryclass = {math},
  publisher = {{arXiv}},
  doi = {10.48550/arXiv.2305.06009},
  urldate = {2023-05-11},
  archiveprefix = {arxiv}
}

@article{peresDomainsAnalyticContinuation1992,
  title = {Domains of analytic continuation for the top {{Lyapunov}} exponent},
  author = {Peres, Yuval},
  year = {1992},
  journal = {Annales de l'I.H.P. Probabilit\'es et statistiques},
  volume = {28},
  number = {1},
  pages = {131--148},
  publisher = {{Gauthier-Villars}},
  urldate = {2023-08-16},
  langid = {english}
}

@article{duarteHolderContinuityLyapunov2022,
  title = {H\"older continuity of the {{Lyapunov}} exponents of linear cocycles over hyperbolic maps},
  author = {Duarte, Pedro and Klein, Silvius and Poletti, Mauricio},
  year = {2022},
  month = dec,
  journal = {Mathematische Zeitschrift},
  volume = {302},
  number = {4},
  pages = {2285--2325},
  doi = {10.1007/s00209-022-03147-9},
  urldate = {2023-08-16},
  langid = {english}
}

@article{furstenbergRandomMatrixProducts1983e,
  title = {Random matrix products and measures on projective spaces},
  author = {Furstenberg, H. and Kifer, Y.},
  year = {1983},
  month = jun,
  journal = {Israel Journal of Mathematics},
  volume = {46},
  number = {1},
  pages = {12--32},
  doi = {10.1007/BF02760620},
  urldate = {2023-08-16},
  langid = {english}
}

@article{hennionDerivabiliteGrandExposant1991,
  title = {D\'erivabilit\'e du plus grand exposant caract\'eristique des produits de matrices al\'eatoires ind\'ependantes \`a coefficients positifs},
  author = {Hennion, Hubert},
  year = {1991},
  journal = {Annales de l'I.H.P. Probabilit\'es et statistiques},
  volume = {27},
  number = {1},
  pages = {27--59},
  urldate = {2023-08-16},
}

@article{hennionLoiGrandsNombres1984,
  title = {Loi des grands nombres et perturbations pour des produits r\'eductibles de matrices al\'eatoires ind\'ependantes},
  author = {Hennion, H.},
  year = {1984},
  month = oct,
  journal = {Zeitschrift f\"ur Wahrscheinlichkeitstheorie und Verwandte Gebiete},
  volume = {67},
  number = {3},
  pages = {265--278},
  doi = {10.1007/BF00535004},
  urldate = {2023-08-16},
  langid = {french}
}

@book{bauerProbabilityTheory2011,
  title = {Probability Theory},
  author = {Bauer, Heinz},
  year = {1996},
  publisher = {De Gruyter},
  address = {Berlin, New York},
  series = {De Gruyter Studies in Mathematics},
  number = 23,
  doi = {10.1515/9783110814668},
  urldate = {2023-08-16},
  isbn = {978-3-11-081466-8},
  langid = {english}
}

@article{shorten2007stability,
  title={Stability criteria for switched and hybrid systems},
  author={Shorten, Robert and Wirth, Fabian and Mason, Oliver and Wulff, Kai and King, Christopher},
  journal={SIAM review},
  volume={49},
  number={4},
  pages={545--592},
  year={2007},
  doi={10.1137/05063516X},
}

@article{margaliot2006stability,
  title={Stability analysis of switched systems using variational principles: an introduction},
  author={Margaliot, Michael},
  journal={Automatica},
  volume={42},
  number={12},
  pages={2059--2077},
  year={2006},
  doi = {10.1016/j.automatica.2006.06.020},
}

@article{mason2014extremal,
  title = {Extremal norms for positive linear inclusions},
  author = {Mason, Oliver and Wirth, Fabian},
  year = {2014},
  month = mar,
  journal = {Linear Algebra and its Applications},
  volume = {444},
  pages = {100--113},
  doi = {10.1016/j.laa.2013.11.020},
}

@article{bousch2002asymptotic,
  title = {Asymptotic height optimization for topical {IFS}, {T}etris heaps, and the finiteness conjecture},
  author = {Bousch, Thierry and Mairesse, Jean},
  year = {2002},
  journal = {Journal of the American Mathematical Society},
  volume = {15},
  number = {1},
  pages = {77--111},
  doi = {10.1090/S0894-0347-01-00378-2},
  publisher = {{American Mathematical Society}},
}

@article{lagariasFinitenessConjectureGeneralized1995,
  title = {The finiteness conjecture for the generalized spectral radius of a set of matrices},
  author = {Lagarias, Jeffrey C. and Wang, Yang},
  year = {1995},
  month = jan,
  journal = {Linear Algebra and its Applications},
  volume = {214},
  pages = {17--42},
  doi = {10.1016/0024-3795(93)00052-2},
}

@article{hareExplicitCounterexampleLagarias2011,
  title = {An explicit counterexample to the {{Lagarias}}{\textendash}{{Wang}} finiteness conjecture},
  author = {Hare, Kevin G. and Morris, Ian D. and Sidorov, Nikita and Theys, Jacques},
  year = {2011},
  month = apr,
  journal = {Advances in Mathematics},
  volume = {226},
  number = {6},
  pages = {4667--4701},
  doi = {10.1016/j.aim.2010.12.012},
}

\appendix

 \section{Proof of \Cref{prop:specrad-lowerbound}}
 \label{sec:App}

 In this appendix we provide a proof for \Cref{prop:specrad-lowerbound}, where we follow the relevant parts of
 \cite[Chapter 2]{katoPerturbationTheoryLinear1995}.
 Similar arguments can be found in \cite[Section 3.7]{baumgartelAnalyticPerturbationTheory1985}.
 In the proof we use the Euclidean norm on $\C^d$ and the 
 induced spectral norm on $\C^{d \times d}$.
\begin{proof}[Proof of \Cref{prop:specrad-lowerbound}] 
  Denote the resolvent of $A$ by $R(\xi):=(A-\xi)^{-1}$
  and define $r(\xi):=\norm{R(\xi)}^{-1}$.
  Let $\lambda$ be an eigenvalue of $A$ with $\abs{\lambda} =
  \rho(A)$ and let $m$ be the multiplicity of $\lambda$.
  Let $\Xi$ be a circle of radius $\delta \in (0,1)$ around $\lambda$
  such that no other eigenvalue of $A$ is contained in the closed disk
  bounded by $\Xi$.
  Set \[r_0:= \min_{\xi \in \Xi} r(\xi).\]
  Set $\varepsilon_0 := \frac{r_0}{2}$ 
  and $\Gamma := \frac{2\delta}{r_0}$.

  Now let $B$ be a matrix with $\norm{B} \leq 1$ and $\varepsilon \in \C$.
  Set $A_\varepsilon:=(A+\varepsilon B)$.
  Denote the resolvent
  of $A_\varepsilon$ by $R(\xi,\varepsilon):=(A_\varepsilon-\xi)^{-1}$.

  Under the above perturbation, the eigenvalue $\lambda$ of $A$
  splits into a group of $m$ eigenvalues (counted
  with multiplicity) of $A_\varepsilon$. The average of these
  eigenvalues is denoted by $\hat{\lambda}(\varepsilon)$
  and satisfies (see \cite[p. 75, Formula (2.5)]{katoPerturbationTheoryLinear1995}
  \begin{align*}
    \hat{\lambda}(\varepsilon) = \frac{1}{m} \tr(A_\varepsilon P(\varepsilon))
  \end{align*}
  where
  $P(\varepsilon) = -\frac{1}{2\pi \iu} \int_\Xi R(\xi,\varepsilon) \,\dd
  \xi$ is the total projection of the $\lambda$-group (see
  \cite[p. 67, Formula (1.16)]{katoPerturbationTheoryLinear1995}.
  This projection can be writen as a power series
  \begin{align*}
    P(\varepsilon) = \sum_{n=0}^\infty \varepsilon^n P^{(n)}.
  \end{align*}
  We can also represent $\hat{\lambda}$ as a power series
  \begin{align*}
    \hat{\lambda}(\varepsilon) = \sum_{n=0}^\infty  \varepsilon^n\lambda^{(n)}
  \end{align*}
  and this power series converges whenever the power series $P(\varepsilon)$
  converges.
  
  We now follow \cite[Section 3.1 Simple
  Estimates]{katoPerturbationTheoryLinear1995}.
  A simple calculation shows
  $R(\xi,\varepsilon)=R(\xi)(I+\varepsilon B R(\xi))^{-1}$.
  For $\abs{\varepsilon}<r_0$ 
  we have $\norm{\varepsilon B R(\xi)}<1$
  for all $\xi \in \Xi$ by the definition of $r_0$.
  Because of this $R(\xi,\varepsilon)$ converges
  as a power series in $\varepsilon$ for all $\xi \in \Xi$
  and hence also the power series $P(\varepsilon)$ and
  $\hat{\lambda}(\varepsilon)$ converge.
  Therefore for $\abs{\varepsilon}<r_0$ the function
  $\hat{\lambda}(\varepsilon)-\lambda$ is holomorphic
  in $\varepsilon$. This function is
  also bounded by $\delta$,
  since otherwise there would
  be $\abs{\varepsilon^*}<r_0$  such that $A_{\varepsilon^*}$
  has an eigenvalue $\lambda^*$ on $\Xi$.
  But then the resolvent $R(\xi,\varepsilon)$ would have a pole in $(\lambda^*,\varepsilon^*)$
  contradicting the fact that $R(\xi,\varepsilon)$
  converges as a power series in $\varepsilon$ for every $\xi \in \Xi$
  as shown above.
  Hence $\abs{\lambda^{(n)}} \leq \frac{\delta}{r_0^{n}}$
  for all $n \in \N$ by Cauchy's inequality,
  see \cite[p. 88, Formula 3.5]{katoPerturbationTheoryLinear1995}.
  For $\abs{\varepsilon} \leq \frac{r_0}{2}$
  we finally
  have
  \begin{align*}
    \abs{\hat{\lambda}(\varepsilon)} &\geq \abs{\lambda}-\sum_{n=1}^\infty
    \abs{\varepsilon}^n \abs{\lambda^{(n)}} \\
    &\geq \abs{\lambda} 
  - \abs{\varepsilon}(\sum_{n=1}^\infty
      \abs{\varepsilon}^{n-1} \abs{\lambda^{(n)}})\\
     &\geq \abs{\lambda} - \abs{\varepsilon}(\sum_{n=1}^\infty (\frac{r_0}{2})^{n-1} \frac{\delta}{r_0^n}) \\
     &\geq \abs{\lambda} - \abs{\varepsilon} \frac{\delta}{r_0}
       \sum_{n=1}^\infty \frac{1}{2^{n-1}} \\
     &=\rho(A)-\abs{\varepsilon} \frac{2\delta}{r_0}.
  \end{align*}
  
  Since at least one of the eigenvalues of $A_\varepsilon$ in the $\lambda$-group
  has modulus at least $\tilde{\lambda}(\varepsilon)$, we finally have
  \begin{align*}
    \rho(A_\varepsilon) \geq \abs{\tilde{\lambda}(\varepsilon)} \geq \rho(A) - \Gamma \varepsilon
  \end{align*}
  for $\varepsilon \in (0,\varepsilon_0)$.
  This completes our proof.
\end{proof}
 
\end{document}